\documentclass[11pt, reqno, a4]{amsart}

\makeatletter
\def\set@curr@file#1{%
  \begingroup
    \escapechar\m@ne
    \xdef\@curr@file{%
      \expandafter\expandafter\expandafter\unquote@name
      \expandafter\expandafter\expandafter{%
      \expandafter\string
        \csname\@firstofone#1\@empty\endcsname}}%
  \endgroup
}
\makeatletter

\usepackage{tikz-cd}
\usepackage{textcmds} 
\usepackage{amssymb}

\usepackage[margin=2.8cm]{geometry}

 \usepackage{graphics,graphicx}  

\usepackage{textcmds}  
\usepackage{amsmath, amssymb, amsfonts, amstext, verbatim, amsthm, mathrsfs}
\usepackage[mathcal]{eucal}
\usepackage{microtype}
\usepackage[all]{xy}
\usepackage[modulo]{lineno}

\usepackage{aliascnt}
\usepackage[shortlabels]{enumitem}
\usepackage{xspace}
\usepackage{amsfonts}
\usepackage{amssymb}
\usepackage[centertags]{amsmath}
\usepackage{amsthm}
\usepackage{dsfont}
\usepackage{bm}
\usepackage{xcolor}
\usepackage{subfigure}
\usepackage{amsmath}
\usepackage{array}
\usepackage[all]{xy}
\usepackage{makecell}

\usepackage{parskip}

\usepackage[backref,colorlinks=true,linkcolor=blue,citecolor=blue,urlcolor=blue,citebordercolor={0 0 1},urlbordercolor={0 0 1},linkbordercolor={0 0 1}]{hyperref} 

\usepackage{amssymb}

\usepackage{hyperref}

\newtheorem{theorem}[equation]{Theorem}

\newtheorem{lemma}[equation]{Lemma}

\newtheorem{prop}[equation]{Proposition}
\newtheorem{proposition}[equation]{Proposition}
\newtheorem{corollary}[equation]{Corollary}

\newtheorem{application}[equation]{Application}
\newtheorem{definition-lemma}[equation]{Definition-Lemma}

\theoremstyle{definition}
\newtheorem{definition}[equation]{Definition}

\newtheorem{example}[equation]{Example}

\theoremstyle{remark}
\newtheorem{remark}[equation]{Remark}

\numberwithin{equation}{section}
\numberwithin{figure}{section}

\newcommand{\bZ} {\mathbb{Z}}
\newcommand{\Z} {\mathbb{Z}}

\newcommand{\bR} {\mathbb{R}}
\newcommand{\R} {\mathbb{R}}
\newcommand{\bC} {\mathbb{C}}
\newcommand{\C} {\mathbb{C}}

\newcommand{\bP} {\mathbb{P}}
\newcommand{\bF} {\mathbb{F}}

\newcommand {\cE}  {\mathcal{E}}
\newcommand {\cF}  {\mathcal{F}}
\newcommand {\cG}  {\mathcal{G}}

\newcommand {\cI}  {\mathcal{I}}
\newcommand {\cJ}  {\mathcal{J}}
\newcommand {\cK}  {\mathcal{K}}

\newcommand {\cO}  {\mathcal{O}}

\newcommand {\cV}  {\mathcal{V}}
\newcommand {\cW} {\mathcal{W}}

\newcommand {\vv} {\mathbf{v}}
\newcommand{\nn}{\mathbf{n}}

\newcommand {\oD} {\bar{D}}

\newcommand {\oY} {\bar{Y}}

\newcommand{\prim}{\mathrm{prim}}

\newcommand {\io} {\iota}

\newcommand{\zm}{\Z^2 \setminus \{0\}}

\renewcommand {\ker} {\operatorname{ker}}

\newcommand {\id}  {\operatorname{id}}
\newcommand {\Id}  {\operatorname{Id}}

\newcommand {\Hom}  {\operatorname{Hom}}
\newcommand {\Ham}  {\operatorname{Ham}}
\newcommand {\Pic}  {\operatorname{Pic}}

\newcommand {\Bl}  {\mathrm{Bl}}

\newcommand {\cst}  {(\C^*)^2}

\newcommand {\SL}  {\operatorname{SL}}
\newcommand {\GL} {\operatorname{GL}}

\DeclareMathOperator{\Perf}{Perf}
\DeclareMathOperator{\coh}{coh}
\DeclareMathOperator{\perf}{perf}
\DeclareMathOperator{\ccoh}{\mathfrak{c}oh}
\DeclareMathOperator{\pperf}{\mathfrak{p}erf}

\DeclareMathOperator{\Auteq}{Auteq}
\DeclareMathOperator{\Aut}{Aut}

\newcommand{\dg}{\mathrm{dg}}

\newcommand{\gr}{\mathrm{gr}}

\newcommand{\sch}{\mathcal{C}}
\newcommand{\prs}{\mathrm{pro}\mathcal{C}}

\newcommand{\into}{\hookrightarrow}

\newcommand{\un}{\mathrm{univ}}
\newcommand{\univ}{\mathrm{univ}}
\newcommand{\Yuniv}{Y_\mathrm{univ}}
\newcommand{\Uuniv}{U_\mathrm{univ}}

\newcommand{\inv}{^{-1}}

\newcommand{\W}{\mathcal{W}}
\newcommand{\Symp}{\operatorname{Symp}}

\newcommand{\Bir}{\operatorname{Bir}}

\newcommand{\CY}{\operatorname{CY2}}

\makeatletter
\newcommand{\holim@}[2]{%
  \vtop{\m@th\ialign{##\cr
    \hfil$#1\operator@font holim$\hfil\cr
    \noalign{\nointerlineskip\kern1.5\ex@}#2\cr
    \noalign{\nointerlineskip\kern-\ex@}\cr}}%
}
\newcommand{\holim}{%
  \mathop{\mathpalette\holim@{\rightarrowfill@\textstyle}}\nmlimits@
}
\makeatother

\newcommand{\white}{\textcolor{white}}

\def\mydate{\ifcase\month \or January\or February\or March\or
April\or May\or June\or July\or August\or September\or October\or 
November\or December\fi \space\number\day,\space\number\year}

\begin{document}

\title{
A universal mirror for $(\bP^2, \Omega)$ as a birational object
}

\author{Ailsa Keating}
\email{amk50@cam.ac.uk}
\address{Department of Pure Mathematics and Mathematical Statistics, Centre for Mathematical Sciences, University of Cambridge, Wilberforce Road, Cambridge, CB3 0WB}
\author{Abigail Ward}
\email{arw204@cam.ac.uk}
\address{Department of Pure Mathematics and Mathematical Statistics, Centre for Mathematical Sciences, University of Cambridge, Wilberforce Road, Cambridge, CB3 0WB}

\maketitle


\begin{abstract} 

We study homological mirror symmetry for $(\bP^2, \Omega)$ viewed as an object of birational geometry, with $\Omega$ the standard meromorphic volume form. First, we construct universal objects  on the two sides of mirror symmetry, focusing on the exact symplectic setting: a smooth complex scheme $U_\univ$ and a Weinstein manifold $M_\univ$, both of infinite type; and we prove homological mirror symmetry for them. Second, we consider autoequivalences. We prove that automorphisms of $U_\univ$ are given by a natural discrete subgroup of $\Bir (\bP^2, \pm \Omega)$; and that all of these automorphisms are mirror to symplectomorphisms of $M_\univ$. We conclude with some applications. 
\end{abstract}
 

\section{Introduction}

Let $\Omega$ denote a meromorphic volume form on $\bP^2$ with simple poles along the standard toric boundary divisor $\Delta$. The guiding question for this article is: can we make sense of a `mirror' to $(\bP^2, \Omega)$ as a object of birational geometry?


Homological mirror symmetry, in its most naive form, is about  isomorphisms between derived categories of coherent sheaves on one side  and Fukaya categories on the other. On the other hand, within algebraic geometry, the relation between derived equivalence and birationality is the subject of much study. This paper is a case study tying the two together in complex dimension two.
Our main theorems, below, give the universal examples in complex dimension two for the two sides of mirror symmetry, when focusing on the exact symplectic setting. 

The pair $(\bP^2, \Omega)$  is a key example. 
Assume that $(Y,D)$ is a log Calabi--Yau (CY) variety, i.e., a pair consisting of a closed smooth complex variety $Y$ together with an anti-canonical divisor $D \subset Y$. 
These are the open counterparts to compact Calabi--Yau varieties. 
In dimension two, any such $(Y,D)$ with $D$ singular has a toric model: through a sequence of blow downs and blow ups of points on $D$, the surface $Y$ is birational to $\bP^2$, with the image of $D$ under the birational equivalence compactifying to the boundary divisor $\Delta$. As such  $(\bP^2, \Omega)$  captures essentially all log Calabi-Yau surfaces. At the same time, such surfaces are an extremely rich class of examples: classical  mirror symmetry for log CY surfaces, which has been comprehensively established, includes a wealth of connections to the theory of surface singularities, and to cluster theory; see inter alia \cite{Auroux, Pascaleff, GHK1, GHK2, GHKK, Friedman, Engel, HK1, Cheung-Vianna, Collins-Jacob-Lin, HK2, Kim}. 

By work of Blanc \cite{Blanc}, we know that $\Bir(\bP^2, \pm\Omega)$ is generated by the subgroup $\Aut(\cst) \simeq \cst \rtimes \GL_2(\Z)$ and a given cluster transformation $E$. We consider the discrete subgroup $\Bir_{e}(\bP^2,  \pm \Omega) = \langle \GL_2(\Z), E \rangle.$ (See below for further discussion of this choice.) 

\begin{theorem} [Theorem \ref{thm:autisos}] There exists a smooth, infinite-type scheme $U_{\univ}$ over $\C$ such that 
\[ \Aut(U_{\univ}) \simeq \Bir_{e}(\bP^2, \pm \Omega). \]
\end{theorem} 

The scheme $U_{\univ}$ is constructed by taking the infinite union of a system of open Calabi--Yau surfaces partially ordered by inclusion, each distinguished within their deformation class by the fact that the mixed Hodge structure on middle homology splits into two pure Hodge structures of different weights. $U_\univ$ can be thought of as the universal object for such Calabi-Yau surfaces. 

On the symplectic side, we have a directed system of Weinstein four-manifolds, again partially ordered by inclusion; the limit of this system is an infinite type Weinstein manifold $M_{\univ}$. 
This can be thought of as the universal object for Milnor fibres of cusp singularities. 

\begin{theorem}[Corollary \ref{cor:universal-exact-hms} and Theorem \ref{thm:mirror-map}]\label{thm:1} There is a quasi-equivalence of $A_\infty$ categories
\[ \perf U_{\univ} \simeq \cW(M_{\univ}) \] 
where $\cW(M_\univ)$ denotes the wrapped Fukaya category of $M_\univ$. 
Moreover, there is an injective group homomorphism  
\[
\Bir_{e}(\bP^2, \pm \Omega) \hookrightarrow  \Symp_e M_\univ  / ( \Ham_c M_\univ ).\\  
\]
Here $\Symp_e$ denotes exact symplectomorphisms, and $\Ham_c$ compactly supported Hamiltonian isotopies. 

The group homomorphism has the property that for any $\phi \in \Bir_e(\bP^2, \pm \Omega)$, its image $\phi^\vee $ has a well-defined action on $\cW(M_\univ)$, and  the following diagram commutes up to $A_\infty$ homotopy:
$$
\xymatrix{
\perf U_{\univ}   \ar[d]^-\simeq \ar[r]_-{\phi_\ast}  &  \perf U_{\univ}  \ar[d]^-\simeq\\
\cW(M_{\univ})   \ar[r]^-{\phi^\vee}  & \cW(M_{\univ})    \\
}
$$
\end{theorem}

\subsection{Mirrors to automorphisms of open log Calabi-Yau surfaces.} As a corollary to Theorem \ref{thm:1} and its proof, we obtain the following result, which extends \cite[Theorem 6.1]{HK2} (which applies to automorphisms of closed Calabi-Yau surfaces which fix the boundary divisor pointwise). 

\begin{corollary}[Corollary \ref{cor:mirror-symplecto-fixed}] Let $(Y,D)$ be a log Calabi-Yau surface with maximally degenerate boundary and distinguished complex structure, let $U= Y \setminus D$, and let $(M, \omega)$ be the Weinstein manifold mirror to $U$ constructed in \cite{HK1}. Further assume that $H_1(U; \Z)=0$.
Then there exists an injective map  
\[
\Aut U \hookrightarrow \pi_0\Symp^\gr (M) 
\]
where $\Symp^\gr$ denotes graded symplectomorphisms.  Moreover, the mirror symmetry correspondence $\coh U \simeq \cW(M)$ intertwines the induced actions of $\Aut(U)$ on the two triangulated categories. \end{corollary}

\begin{application}[Proposition \ref{prop:cubic-incl}]Let $U_C$ be the cubic surface with split mixed Hodge structure, constructed by blowing up each of three preferred points on the toric boundary divisor of $\bP^2$ twice, and let $M_C$ denote its mirror. There is an inclusion 
\[ \Z/2\Z \star \Z/2 \Z \star \Z/2 \Z \hookrightarrow \pi_0 \Symp^\gr (M_C). \]
\end{application}

\subsection{Mirrors to general elements of $\Bir(\bP^2, \pm \Omega)$} 
Why restrict ourselves to $\Bir_{e}(\bP^2, \pm \Omega)$? Recall that the full group $\Bir(\bP^2, \pm \Omega)$ is generated by $\Bir_{e}(\bP^2, \pm \Omega)$ together with $(\bC^\ast)^2$. 
Our primary interest is automorphisms mirror to symplectomorphisms (i.e.~autoequivalences which are `geometric'). The $(\bC^\ast)^2$ belongs to a different regime. Its action on the Fukaya category is not geometric in the most basic example: in the case of $\bP^2 \backslash \Delta$, $(\bC^\ast)^2$ has a mirror action on the Fukaya category by reparametrisations of  local systems on Lagrangians (see Remark \ref{rem:non-symplectos}).

Some elements of  $\Bir(\bP^2, \pm \Omega)$  fail to be regular on any open CY surface with split mixed Hodge structure: for instance, the map \[(x,y) \mapsto (x, (x+1)\inv(x+\lambda)\inv y), \, \lambda \neq 1,  \] which is indeterminate at $[-1:0:1]$ and $[-\lambda: 0 :1]$. 
To resolve such maps, one needs to allow surfaces with non-split mixed Hodge structures. There is a well developed expectation that such deformations of the complex structure should be mirror to non-exact deformations of the symplectic form, together with the introduction of a $B$-field; as above with the toy case of $(\bC^\ast)^2$, the mirror automorphisms of deformed Fukaya categories 
are not in general expected to be geometric.

\subsection*{Acknowledgements}
We thank Philip Engel, Sheel Ganatra, Fran\c{c}ois Greer, Paul Hacking, Andrew Hanlon, Nick Rozenbluym, and Jan Steinebrunner for helpful conversations, and J{\'a}nos Koll{\'a}r for correspondence. We also thank the anonymous referee for their careful reading of this paper and insightful comments and suggestions.

AK was partially supported by EPSRC Fellowship EP/W001780/1 and ERC Starting Grant `SingSymp' (grant number ~101041249). AW was supported by NSF Grant DMS-2002183 and  UKRI Frontier Research Grant `Floer Theory Beyond Floer’ (grant number EP/X030660/1). 

\textbf{Open Access.} For the purpose of open access, the authors have applied a Creative Commons Attribution (CC:BY) licence to any Author Accepted Manuscript version arising from this submission.
\textbf{UKRI data access statement.} 
There is no dataset associated with this paper.

\section{Universal objects} \label{sec:univ-obj-exact}

\subsection{Systems of log Calabi-Yau surfaces with split mixed Hodge structure}\label{sec:bsidesystems}
Unless otherwise specified, all log Calabi-Yau surfaces $(Y,D)$ throughout will be assumed to have maximally degenerate boundary. In this dimension, this means that $D$ has at least one node. We assume moreover that $(Y,D)$ corresponds to the unique point within its complex deformation class such that the mixed Hodge structure on $H_2 (Y \backslash D, \bZ)$ splits. In the language of \cite{HK1}, $(Y,D)$ has the `distinguished' complex structure within its deformation class. See 
 \cite[Section 2.2]{HK1} for background  exposition.

Fix such a log Calabi-Yau surface $(Y,D)$. By \cite[Proposition 1.3]{GHK2}, this has a toric model: possibly after some corner blow ups (i.e.~blow ups of nodal points of $D$), $(Y,D)$ is given by starting with a toric surface, say $(\oY,\oD)$, and iteratively blowing up interior (i.e.~non-nodal) points on the irreducible components $\oD_i$ of $\oD$. The split mixed Hodge structure condition essentially determines which point gets blown up for each component. Explicitly, for a fixed embedding of the torus $(\bC^\ast)^2$ into $\bar{U} = \bar{Y} \backslash \bar{D}$, take all blow-ups to be at the ``$-1$ points'' of the $\bar{D}_i$: if $\bar{D}_i$ corresponds to the toric ray $(0,1)^T$, we blow up at the limit of $(-1, z)$ as $z \to 0$; and similarly for other rays after conjugating by $\SL_2(\bZ)$ elements.

In this setting, a log Calabi-Yau surface $(Y,D)$ is uniquely determined by two sequences of numbers: the self-intersections $n_i$ of the $\oD_i$, which determine the toric surface, and $m_i$, the number of interior blow-ups performed on the boundary component $\oD_i$. In order to define a universal object for this family, it will be useful for us to, instead of just specifying the $n_i$,  fix an embedding of the toric fan in a reference copy of $\bR^2$. This determines an explicit embedding of $(\C^*)^2$ into the interior $U = Y \backslash D$ and lets us index the irreducible components of the boundary $D$ by primitive elements $\mathbf{n} =(n_1, n_2)$ in $\Z^2$; if we write the embedding as $\iota: (\C^*)^2 \into U$, we have 
\begin{equation} \label{eq:Dn} D_{\nn} = \{ x \in Y \mid x = \lim_{t \to 0} \, \iota (\lambda_1 t^{n_1}, \lambda_2 t^{n_2}) \text{ for some } (\lambda_1, \lambda_2) \in (\C^*)^2 \}. \end{equation} 
The map $ \lambda \mapsto \lim_{t \to 0} \iota( \lambda^{-n_2} t^{n_1}, \lambda^{-n_1} t^{n_2}) $
defines a canonical isomorphism $\C^* \cong D_{\nn}^{\mathrm{int}}$.

\begin{definition}
A log Calabi--Yau surface $(Y,D)$ \emph{with an explicit toric model} is given by the data of a smooth toric fan in 
$\bZ^2$, together with a non-negative integer $m_{\mathbf{n}}$ for each ray of the fan, which is the number of blow ups at the distinguished point of $\oD_\mathbf{n}$.
We record the integers by defining a subset $S$ of $\zm$: 
\begin{equation}\label{eq:Sdefn} S= \bigcup_{\{\nn \in \zm \mid m_\nn \neq 0\}} \{ j \nn\}_{j=1}^{m_{\nn}}. \end{equation}
\end{definition}


\begin{definition}
Let $\CY$ be the set of log Calabi-Yau surfaces with explicit toric models. Given $\xi \in \CY$, let $(Y_\xi, D_\xi)$ denote the corresponding log CY surface with explicit toric model, and set $U_\xi = Y_\xi \backslash D_\xi$.  We let $S_{\xi} \subset \zm$ denote the set indexing the interior blow-ups in the explicit toric model for $Y_{\xi}$.  Note that if $\xi$ and $\xi'$ index surfaces which can be related by a sequence of corner blow ups, both $U_\xi=U_{\xi'}$ and $S_{\xi}=S_{\xi'}$. 

We endow $\CY$ with a partial ordering according to the blow-up relation, so if $\xi, \eta \in  \CY$ satisfy $\eta \geq \xi$, we have a map 
 \[ p_{\eta, \xi}: (Y_\eta, D_\eta) \to (Y_\xi, D_\xi)\]
 which can be factorized as a sequence of corner and interior blow-downs, as well as a natural inclusion map 
 \[ i_{\xi, \eta}: U_{\xi} \into U_{\eta}. \]

\end{definition}

\begin{definition}
We define $\Yuniv$, the universal distinguished log Calabi-Yau surface, by taking the inverse limit under blow ups:
$$\Yuniv = \varprojlim_{\xi \in \CY} Y_\xi.$$
This is not a scheme, but it is an object in the pro-completion of the category of schemes (see, e.g., \cite{Lurie} for the definition of the pro-completion of a category which admits finite limits).
We define the universal distinguished open Calabi-Yau surface, $\Uuniv$, by taking the direct limit under inclusions:
$$\Uuniv= \varinjlim_{\xi \in \CY} U_\xi.$$
This is a scheme which is locally of finite type.
\end{definition}

\subsection{Perfect and coherent sheaves} 
Recall that if a scheme $X$ is quasi-projective, the bounded derived category of coherent sheaves $D(X)$, and the derived category of perfect complexes $\Perf X$, each have unique dg enhancements up to quasi-isomorphism  \cite[Theorems 2.13, 2.14]{Lunts-Orlov}.
(Of course, if $X$ is smooth and quasi-compact, we have $\Perf X =  D(X)$.) Also, if $X$ is projective and e.g.~smooth, the dg enhancement of $D(X)$ is in fact strongly unique, see \cite[Definition 2.3 and Theorem 2.14]{Lunts-Orlov}.

More generally, for an arbitrary scheme $X$, one can define its dg categories of quasi-coherent sheaves, and its dg category of perfect sheaves \cite{CaSt}.
(For an $\infty$-categorical treatment see also ~\cite{GR1}, specifically \cite[Chap.~3, Section 3.6]{GR1} for perfect sheaves.) 

\subsubsection{A dg model for $\Perf U_\univ$} The goal of this subsection is to construct an explicit dg enhancement of $\Perf U_\univ$, which we will denote  $\perf U_\univ$.

Let $U$ denote either $U_{\xi}$ for some $\xi \in \CY$ or the infinite-type scheme $U_{\un}$, and let $C(\mathrm{Q} \coh U)$ be the category of complexes of quasi-coherent sheaves on $U$. Write $C^{\mathrm{fib}}(\mathrm{Q} \coh U)$ for the full dg subcategory consisting of fibrant objects, i.e., h-injective complexes $\cK$ where each term $\cK^n$ is injective, and let $\perf_{\dg}(U)$ denote the full subcategory of $C^{\mathrm{fib}} (\mathrm{Q} \coh U)$ generated by the perfect objects. (See \cite[Section 3.1]{LuSc} for background.)

Recall that for any two $\cJ, \cK \in \perf_{\dg}(U)$, the h-injectivity of the complexes implies that the complex $\Hom(\cJ, \cK)$ computes the homs in the derived category. In general one has: 

\begin{proposition} For any $U$ under consideration, the dg categories $\perf_{\dg}( U)$ are dg enhancements of the derived category $\Perf U$. \end{proposition}
\begin{proof} This is stated in \cite[Section 3.1]{LuSc} for quasi-projective $U$, and for $U_{\univ}$ the same conclusion follows from Gabber's result that $\mathrm{Q}\coh X$ is a Grothendieck category  \cite[Proposition \href{https://stacks.math.columbia.edu/tag/077P}{077P}]{Stacks} for any scheme $X$, together with the existence of functorial fibrant resolutions \cite[Theorem \href{https://stacks.math.columbia.edu/tag/079P}{079P}]{Stacks}, \cite{Serpe}. \end{proof}

\begin{definition} Define $\coh U_{\xi} := \perf_{\dg}(U_\xi)$ and $\perf U := \perf_{\dg} (U).$ 
\end{definition} 

Now let $\xi, \eta \in \CY$ be such that $\eta \geq \xi$ and let $i=i_{\xi,\eta}: U_{\xi} \into U_{\eta}$ be the inclusion. Since fibrant/perfect  complexes remain fibrant/perfect under restrictions to open subsets, there is a natural pullback map 
\[ i^*: \coh U_\eta \to \coh U_{\xi}\] 
computing the pullback on the derived categories. For any chain $\alpha < \beta < \gamma$ in $\CY$ it holds that 
$i_{\beta,\gamma}^* \circ i_{\alpha,\beta}^* = i_{\alpha,\gamma}^*$
 so these maps define a strictly commuting diagram of dg categories.

\begin{lemma} \label{lem:Usheaves} There is a quasi-isomorphism
\begin{equation*}  \perf U_\un \cong \varprojlim_{\xi \in \CY} \coh U_\xi. \end{equation*} \end{lemma}
\begin{proof} The above limit is the strict limit of the diagram $\{\coh U_{\xi} \}_{\xi \in U_{\xi}}$:  objects are elements of 
\[  \left\{ \left(\cJ_{\xi}\right) \in \prod_{\xi \in \CY} \coh U_{\xi} \,\, \bigg|\,  \cI_{\eta}|_{U_\xi} = \cI_{\xi} \text{ for all } \eta > \xi \right\} \] 
and morphisms are given by 
 \[ \Hom((\cJ_{\xi}), (\cK_{\xi}))= \left\{ (a_{\xi}) \in \prod_{\xi \in \CY} \Hom_{\coh U_{\xi}}(\cJ_{\xi}, \cK_{\xi})  \, \bigg|\ a_{\eta}|_{U_{\xi}} = a_{\xi}  \text{ for all } \eta > \xi \right\}. \] 
The restriction functors $i_{\xi}^*: \perf U \mapsto \coh U_{\xi}$ induce a natural functor 
\begin{equation} \label{eq:perfu} i^*:  \perf U_\un \to \varprojlim_{\xi \in \CY} \coh U_\xi. \end{equation}
This is easily seen to be fully faithful: the fact that an element $(a_\xi) \in \Hom((\cJ_{\xi}), (\cK_{\xi}))$ glues to a global section of $\mathcal{H}\kern -.5pt om(\cJ, \cK)$ follows from the fact that any two $U_{\xi_1}, U_{\xi_2}$ embed into a common $U_{\eta}$ on which we can check compatibilities. 

Moreover, any object in the limit provides the sheaf data required to specify some complex $\cK$ on $U_{\univ}$. The complex $\cK$ is perfect since perfectness can be checked locally on quasi-compact opens, and thus quasi-isomorphic to some fibrant object $\cK'$. The essential surjectivity of $i^*$ follows.
 \end{proof}



\subsubsection{Coherent sheaves on $\Yuniv$}

Next, we want to define a (triangulated) category of coherent sheaves on $\Yuniv$, which will be the homotopy category of a dg enhancement $\coh Y_\univ$.
To do so, we switch to the \v{Cech} model appearing in \cite[Section 2.4]{kuznetsov} to define our  dg enhancements $\coh Y_{\xi}$. This model has the advantage that it behaves nicely under pullbacks.
 
 The construction can be summarised as follows. For a single $Y$ in our system, we choose a finite affine  cover $\cV=\{V_i\}_{i \in I}$ of $Y$, and for any $ \mathbf{i}=(i_1, \ldots, i_p) \in I^p$, we set 
$V_{\mathbf{i}} = V_{i_1} \cap \cdots \cap V_{i_p}.$ 
To this cover we can associate the \v{C}ech cosimplicial category $\mathscr{C}_{Y, \cV}$: the objects of $\mathscr{C}_{Y, \cV}$ are finite complexes of vector bundles on $Y$, and for any two complexes $\cF$ and $\cG$, there is an associated cosimplicial complex of vector bundles $\Hom_{\mathscr{C}_{Y, \cV}}^\bullet(\cF,\cG)$ defined by 
\[ \Hom_{\mathscr{C}_{Y, \cV}}^p(\cF,\cG)= \bigoplus_{\mathbf{i} \in I^{p+1}} \Gamma(V_{\mathbf{i}}, \cF^\vee \otimes \cG). \]
Recall that given a cosimplicial object $A^\bullet$ in an abelian category $\mathscr{A}$, there is an associated `normalised' object $N(A^\bullet) \in \mathrm{Ch}_{\geq 0}(\mathcal{A})$ (the relevant background appears in \cite[Section 2.3]{kuznetsov}). In our case, we obtain a bigraded complex $N^{*, *} (\Hom_{\mathscr{C}_{Y, \cV}}(\cF,\cG))$ of vector spaces, where the bigrading is by (internal degree of $\cF^\vee \otimes \cG$, \v{C}ech cosimplicial degree).  We denote the totalisation of this complex by $\check{C}_N^*(\cV, \cF^\vee \otimes \cG)$. (See Remark \ref{rem:norm} for a justification of this notation.)

We now consider the normalised category $N \mathscr{C}_{Y, \cV}$ and its homotopy category $[N \mathscr{C}_{Y, \cV}]$. The former is a dg category with the same objects as $\mathscr{C}_{Y, \cV}$ and with Hom spaces defined by
\begin{equation} \label{eqn:nhoms} \Hom_{N \mathscr{C}_{Y, \cV}}^*(\cF,\cG):= \check{C}_N^*(\cV, \cF^\vee \otimes \cG).  \end {equation} 
The above complex computes the hypercohomology of $\mathcal{H}\kern -.5pt om(\cF,\cG)= \cF^\vee \otimes \cG$, and it follows that $N \mathscr{C}_{Y, \cV}$ is a dg enhancement of $D(Y)$. 

Given one of the divisors $D_{\xi}$ equipped with a \v{C}ech cover $\cV_{\xi}$, we can define a dg category $N \mathscr{C}_{D_{\xi}, \cV_{\xi}}$ in exactly the same way; since $D_{\xi}$ is singular, this will define a dg enhancement of  $\Perf D_{\xi}$. 

\begin{remark}[Comparisons of \v{C}ech complexes]\label{rem:norm} Given any cosimplicial complex of vector spaces $F^\bullet$, we can also consider the `un-normalized' or `simple' bigraded complex $s(F)$. There is a functorial inclusion $N (F) \to s(F)$  which restricts to a quasi-isomorphism
$N^{j,*} F  \to s^{j,*}F $
for all $j$ (this part of the Dold-Kan correspondence, as stated e.g.~in \cite[Theorem 2.12]{Conrad}), and thus a functorial inclusion $\mathrm{Tot}\, N(F) \to \mathrm{Tot} \,s(F)$ which is also a quasi-isomorphism. 
In our case, $s(\Hom_{\mathscr{C}_{Y, \cV}} (F,G))$ is the standard bigraded \v{C}ech complex computing the coherent cohomology of $\mathcal{H}\kern -.5pt om(\cF,\cG) $ according to the cover $\cV$ \cite[Section  \href{https://stacks.math.columbia.edu/tag/01FP}{01FP}]{Stacks}. Denote the totalisation of this complex by $\check{C}_{\mathrm{unord}}^*(\cV, \cF^\vee \otimes \cG)$. We have a sequence 
\begin{equation} \check{C}_N^*(\cV, \cF^\vee \otimes \cG)\to \check{C}_{\mathrm{unord}}^*(\cV, \cF^\vee \otimes \cG) \to R \Gamma(Y, \cF^\vee \otimes \cG)\end{equation}
where the latter map is the canonical quasi-isomorphism defined e.g.~in \cite[Lemma \href{https://stacks.math.columbia.edu/tag/08BN}{08BN}] {Stacks}. 
The enhancement map $H^n (\Hom_{N \mathscr{C}_{Y, \cV}}) \to \Hom_{D(Y)} (\cF,\cG[n])$ factors as the composition of the induced maps on homology (in fact, one could define the enhancement map this way). 
\end{remark} 

\subsubsection*{Refinements} Now suppose $\cW=\{W_j\}_{j \in J}$ is another finite affine open cover of $Y$ which refines the cover $\cV$, that is, there is a map $c: J \to I$ of the indexing sets so that 
$W_j \subset V_{c(j)} $
for all $j \in J$. 
There is a natural (strict) refinement functor $r: \mathscr{C}_{Y, \cV} \to \mathscr{C}_{Y, \cW}$. This is the identity on objects and is defined on Hom spaces by the standard \v{C}ech refinement maps $\Hom_{\mathscr{C}_{Y, \cV}}^\bullet(\cF,\cG)\to \Hom_{\mathscr{C}_{Y, \cW}}^\bullet(\cF,\cG)$, so $ r(\alpha)_{j_0, \cdots, j_{p+1}}= \mathrm{res}_{W_{j_0 \cdots j_n}}(\alpha_{c(j_0), \cdots, c(j_{p+1})}) $
for any $\alpha \in \Hom_{\mathscr{C}_{Y, \cV}}^p(\cF,\cG)$. See also the discussion appearing in \cite[Section \href{https://stacks.math.columbia.edu/tag/01FP}{01FP}]{Stacks}
of the induced refinement maps between the un-normalised complexes.

\subsubsection*{Pullbacks} Let $\eta \geq \xi$ in $\CY$, and let $p:Y_{\eta} \to Y_{\xi}$ be the blowdown morphism. Suppose we have chosen some finite affine open cover $\cV=\{V_i\}_{i \in I}$ of $Y_{\xi}$. There is a natural (strict) pullback functor 
$\mathscr{C}_{Y_{\xi}, \cV} \to  \mathscr{C}_{Y_{\eta}, p\inv\cV}$  
defined on objects by $\cF \mapsto p^*\cF$ and on morphisms by
$ (\alpha_{j}) \mapsto (p^*\alpha_{j}).$ 
The open cover $p\inv \cV$ will in general \emph{not} be affine as some open sets $p\inv(V_i)$ will contain the exceptional divisors contracted by $p$, but we can always further refine it to a finite affine cover $\cW$. After composing with the refinement operation defined above we obtain a functor $(p_{\eta, \xi})^*: \mathscr{C}_{Y_{\xi},\cV} \to \mathscr{C}_{Y_{\eta},\cW}$. (The notation is justified by Proposition \ref{prop:pullback} below.)

\subsubsection*{Pullback functors on dg categories} The normalisation procedure is functorial, so any strict functor $\mathscr{C} \to \mathscr{C}'$ of cosimplicial categories induces a dg functor $N \mathscr{C} \to N \mathscr{C}'$. 

\begin{proposition} \label{prop:pullback} Let $\eta,\xi \in \CY$ be such that $\eta \geq \xi$ and let $p: Y_{\eta} \to Y_{\xi}$ be the blowdown map. Further suppose that $\cV$ is a choice of \v{C}ech cover for $Y_{\xi}$ and let $\cW$ be any \v{C}ech cover of $Y_{\eta}$ refining the cover $p \inv \cV$. Let $p^*: N \mathscr{C}_{Y_\xi, \cV} \to N\mathscr{C}_{Y_\eta, \cW}$ be the functor on dg categories induced by the pullback map defined above. Then the map $p^*: [N \mathscr{C}_{Y_\xi, \cV}] \to [N \mathscr{C}_{Y_\eta, \cW}]$ on homotopy categories 
agrees with the pullback map $p^*: D (Y_{\xi}) \to D ( Y_{\eta})$ under our equivalence. 
 \end{proposition}
  \begin{proof} By \cite[Lemma \href{https://stacks.math.columbia.edu/tag/01FD}{01FD}]{Stacks} and \cite[Lemma \href{https://stacks.math.columbia.edu/tag/08BN}{08BN}]{Stacks} combined with  Remark \ref{rem:norm}, we have a commutative diagram
 \[ \begin{tikzcd}
 \check{C}_N^*(\cV, \cF^\vee \otimes \cG) \arrow[d, "p^*"]  \arrow[r, hook]  & \check{C}_{\mathrm{unord}}^*(\cV, \cF^\vee \otimes \cG) ) \arrow[d, "p^*"] \arrow[r]   &  R \Gamma(Y, \cF^\vee \otimes \cG)\  \arrow[d, "p^*"]   \\
 \check{C}_N^*(\cW, p^*(\cF^\vee \otimes \cG))  \arrow[r, hook]                         &  \check{C}_{\mathrm{unord}}^*(\cW, p^*(\cF^\vee \otimes \cG)) \arrow[r] & R \Gamma(Y, p^*(\cF^\vee \otimes \cG))
\end{tikzcd}
\]
 for any two finite complexes $\cF$ and $\cG$ of vector bundles on $Y_{\xi}$. Here the leftmost vertical arrow is from the induced map on normalised complexes, the middle arrow is the induced map on simple complexes, and the rightmost vertical map is the standard pullback map.  The result follows after passing to cohomology. \end{proof}
  
\begin{lemma}\label{lem:cechcovers} There exists a set of choices of finite affine \v{C}ech covers $\{\cV_{\xi}\}_{\xi \in \CY}$ which are \emph{compatible} in the sense that for all $\eta > \xi$, the cover $\cV_{\eta}$ refines $p_{\eta, \xi}\inv \cV_{\xi}$, and moreover, the pullback functors $(p_{\eta, \xi})^*: N \mathscr{C}_{Y_{\xi}, \cV_{\xi}} \to N \mathscr{C}_{Y_{\eta}, \cV_{\eta}}$
 define a strictly commuting diagram of dg categories indexed by $\CY$. 
 
Furthermore, given any two such choices of compatible covers $\{\cV_{\xi}\}_{\xi \in \CY}$ and $\{\cV_{\xi}'\}_{\xi \in \CY}$, there exists a third set of compatible covers $\{\cW_{\xi}\}_{\xi \in \CY}$ simultaneously refining both and inducing maps of diagrams 
\[ \{N \mathscr{C}_{Y_{\xi}, \cV_{\xi}}\}_{\xi \in \CY} \rightarrow \{N \mathscr{C}_{Y_{\xi}, \cW_{\xi}}\}_{\xi \in \CY} \leftarrow \{N \mathscr{C}_{Y_{\xi}, \cV_{\xi}'} \}_{\xi \in \CY} \]
where the maps $N \mathscr{C}_{Y_{\xi}, \cV_{\xi}} \rightarrow N \mathscr{C}_{Y_{\xi}, \cW_{\xi}} \leftarrow N \mathscr{C}_{Y_{\xi}, \cV_{\xi}'}  $ 
are quasi-equivalences for all $\xi \in \CY$. 
\end{lemma}

\begin{proof} Given a set of finite affine covers $\{\cV_{\xi}\}_{\xi \in \CY}$, each indexed by sets $I_{\xi}$, we claim that the desired compatibility condition is satisfied as soon the assignments $\xi \mapsto I_{\xi}$ and the refinement maps $c_{\eta, \xi}: I_{\eta} \to I_{\xi}$ define a functor 
$I: \CY^{\leq} \to \mathrm{FinSet}$
from the poset category $\CY^{\leq}$ to the category of finite sets. Indeed, in this case, for any chain $\alpha \leq \beta \leq \gamma$ of elements in $\CY$,  the equality
\begin{equation} \label{eqn:schains} c_{\gamma, \beta} \circ c_{\beta, \alpha}=c_{\gamma, \alpha}, \end{equation}
implies that for any objects $\cF,\cG$ of $N \mathscr{C}_{Y_{\xi_1}, \cV_{\xi}}$, the composition 
\[ \Hom_{N \mathscr{C}_{Y_{\alpha}, \cV_{\alpha}}}(\cF,\cG) \to \Hom_{N \mathscr{C}_{Y_{\beta}, \cV_{\beta}}}(p_{\beta, \alpha}^*\cF, p_{\beta, \alpha}^*\cG) \to \Hom_{N \mathscr{C}_{Y_{\gamma}, \cV_{\gamma}}}(p_{\gamma, \alpha}^*\cF,p_{\gamma, \alpha}^*\cG)  \] 
of the pullback functors agrees with the pullback functor
\[ \Hom_{N \mathscr{C}_{Y_{\alpha, \cV_{\alpha}}}} (\cF,\cG) \to \Hom_{N \mathscr{C}_{Y_{\gamma}, \cV_{\gamma}}}(p_{\gamma, \alpha}^*\cF,p_{\gamma, \alpha}^*\cG).  \]

To construct our desired compatible covers, we first independently choose finite affine covers $\cV_{\xi}^o=\{V_{\xi, i}^o\}_{i \in I_{\xi}}$ for all $\xi \in \CY$. Then for each $\xi \in \CY$, define an indexing set 
\[ J_{\xi} = \prod_{\nu \leq \xi} I_{\nu}, \]
so, e.g., if $\xi_0$ is a minimal element of $\CY$, then $J_{\xi_0}=I_{\xi_0}$. 
Note that $J_{\xi}$ is finite since there are only finitely many $\nu$ below $\xi$.  For $j=(j_{\eta})_{\nu \leq \xi} \in J_{\xi}$, define 
\[ V_{\xi, j} = \bigcap_{\eta \leq \xi} p_{\eta, \xi}\inv(V_{\nu, j_{\nu}}^o). \]
It is a general fact that if $f: X \to Y$ is a morphism between two separated schemes, and $U$ and $V$ are affine in $X$ and $Y$ respectively, then $f\inv(V) \cap U = V \times_{Y} U$ is affine \cite[Prop. 5.5.10]{EGA1}; we can apply this fact to conclude that  $V_{\xi,j}$, which is built out of a sequence of such intersections, is affine itself, so $\cV_{\xi} = \{V_{\xi, j}\}_{j \in J_{\xi}}$ defines a new finite affine cover of $Y_{\xi}$. 

For any pair $\eta, \xi \in \CY$ with $\eta > \xi$, there is a natural projection map $J_{\eta} \to J_{\xi}$. These projection maps satisfy the equality in Equation \ref{eqn:schains}, and moreover, define refinements between the affine covers.  We conclude that the system of refinements $\{\cV_{\xi}\}_{\xi \in \CY}$ satisfies our desired conditions. 

Now suppose we have two such choices $\{\cV_{\xi}\}_{\xi \in \CY}$ and $\{\cV_{\xi}'\}_{\xi \in \CY}$. For each $\xi \in \CY$, let
$J_{\xi} = I_{\xi} \times I_{\xi}',$
and define, for each $j=(i, i') \in J_{\xi}$, a new finite affine cover $\cW_{\xi}$ indexed by $J_{\xi}$ with
$W_{\xi, j} = V_{\xi, i} \cap V_{\xi, i'}. $
For $\eta > \xi$, the map $c_{\eta, \xi} \times c_{\eta, \xi}'$ defines a refinement map $J_{\eta} \to J_{\xi}$, and these maps induce a diagram of dg enhancements $\{N \mathscr{C}_{Y_{\xi}, \cV_{\xi}}\}_{\xi \in \CY} \rightarrow \{N \mathscr{C}_{Y_{\xi}, \cW_{\xi}}\}_{\xi \in \CY}$.

The projection maps $J_{\xi} \to I_{\xi}$ define a refinement of $\cV_{\xi}$ to $\cW_{\xi}$ for each $\xi$, and these define a diagram of categories 
\[ \{N \mathscr{C}_{Y_{\xi}, \cV_{\xi}}\}_{\xi \in \CY} \rightarrow \{N \mathscr{C}_{Y_{\xi}, \cW_{\xi}}\}_{\xi \in \CY} \]
restricting to a dg quasi-equivalence for any $\xi \in \CY$. To conclude we note that there is a similar diagram map from $\{N \mathscr{C}_{Y_{\xi}, \cV_{\xi}'}\}_{\xi \in \CY}$ as the construction is symmetric on both sides. 
\end{proof} 

%

\begin{definition} \label{def:y} Given a choice of compatible  \v{C}ech covers, we fix dg enhancements 
\[ \coh Y_\xi := N \mathscr{C}_{Y_{\xi}, \cV_{\xi}}. \]
We then define
$$ \coh \Yuniv: = \holim_{\xi \in \CY} \coh Y_\xi
$$
using the model of the homotopy colimit of a diagram of $A_{\infty}$ categories appearing in \cite[Appendix A]{GPS2}. This gives us an $A_{\infty}$ category (in fact, a dg category, but this doesn't matter for our purposes). The second statement in Lemma \ref{lem:cechcovers} combined with  \cite[Lemma A.8]{GPS2} ensures that  $\coh \Yuniv$ is well-defined up to $A_{\infty}$ quasi-equivalence. 
 (As it contains the limit of the structure sheaves, which has infinite dimensional self-homs, we choose the notation $\coh \Yuniv$ over $\perf \Yuniv.$)
\end{definition}

\begin{definition} \label{def:d} Given a compatible \v{Cech} cover $\{\cV_{\xi}\}_{\xi \in \CY}$ of the system of boundary divisors $\{D_{\xi}\}_{\xi \in \CY}$ (obtained, e.g., by restricting one of the compatible covers constructed above), we define 
\[ \perf D_\xi := N \mathscr{C}_{D_{\xi}, \cV_{\xi}}. \]

After fixing a system of dg enhancements of $\perf D_{\xi}$, we define 
$$ \perf D_{\un} := \holim_{\xi \in \CY} \perf D_\xi. 
$$
\end{definition}
%

\subsection{Homological mirror symmetry for $(Y_\xi, D_\xi)$: review}

Homological mirror symmetry for $(Y_\xi, D_\xi)$ was established in \cite{HK1}. We give an overview of what we will use. For all the variants of the Fukaya category that we consider, assume throughout that we are working with the twisted complex enhancements, and using $\bZ$ gradings and $\bC$ coefficients.
Our starting point is the following:

\begin{theorem}\cite{HK1}\label{thm:original-hms}
Fix $\xi \in \CY$. Then there exists a Weinstein manifold $M_\xi$ and a Weinstein Lefschetz fibration $w_\xi: M_\xi \to \bC$, with smooth fibre $\Sigma_\xi$, such that we have the following triple of $A_\infty$ quasi-isomorphisms:
\begin{eqnarray}
\cF (\Sigma_\xi) & \simeq & \pperf D_\xi \\
 \cF^\to (w_\xi )& \simeq & \ccoh Y_\xi \\
\cW(M_\xi) & \simeq & \ccoh U_\xi
\end{eqnarray}
\end{theorem}
Here for each $\xi$, $\ccoh (Y_\xi)$, respectively $\ccoh U_\xi$, and $\pperf D_\xi$ are explicit $A_\infty$ enhancements of $D(Y_\xi)$, respectively $D(U_\xi)$ and $\Perf D_\xi$.
There are also explicit models for the $A$-side categories, along with collections of generators on each side, and a preferred equivalence matching them up. 

%
%

By Lunts-Orlov, $\ccoh (Y_\xi)$, respectively $\ccoh U_\xi$ and $\pperf D_\xi$, are abstractly $A_\infty$ quasi-isomorphic to, respectively, $\coh Y_\xi$, $\coh U_\xi$ and $\perf D_\xi$ \cite[Theorems 2.12]{Lunts-Orlov}. 
Say the second HMS equivalence above is $i_\xi: \cF^\to (w_\xi ) \to \ccoh Y_\xi$.  Strong uniqueness of dg enhancements of $D(X)$ for a smooth projective $X$ \cite[Theorem 2.14]{Lunts-Orlov} implies that we can fix an $A_\infty$ quasi-isomorphism $s_\xi:  \ccoh Y_\xi \to \coh  Y_\xi$ so that $H^0 ( i_\xi) = H^0 (s_\xi \circ i_\xi): H^0(\cF^\to (w_\xi )) \to  D(Y_\xi) $. 
Similarly for the other two HMS isomorphisms, by considering restriction/localisation maps between the different categories.  (For $U_\xi$, the dg category $\coh Y_\xi$ localises to a third enhancement of $D(U_\xi)$, say $\mathnormal{c}\mathrm{oh} \, U_\xi$, defined in terms of the restriction of the \v{Cech} cover on $Y_\xi$, and we also use the fact that the natural dg functor $\mathnormal{c}\mathrm{oh} \, U_\xi \to \coh U_\xi$ induces an isomorphism of homotopy categories.)

We will want compatibilities between these HMS quasi-isomorphisms within the directed system $\CY$. In order to establish these, we first briefly recall features from \cite{HK1}.

\subsubsection{Weinstein handlebody description}
Given $\xi \in \CY$, $M_\xi$ can be described, first, as a Weinstein handlebody. This is given by starting with the disc bundle $D^\ast T^2$, and, for each ray $\R_{\geq 0} \mathbf{n}$ in the toric fan, attaching $m_\mathbf{n}$ Weinstein 2-handles along the Legendrian $S^1$ in $\partial D^\ast T^2$ which is the conormal lift of $\R \mathbf{n}$ with its orientation (here $T^2 = ( \bZ^2 \otimes \bR) / \bZ^2$).

\subsubsection{Lefschetz fibration $w_\xi$}
Secondly, $\xi \in \CY$ determines a Weinstein Lefschetz fibration $w_\xi: M_\xi \to \bC$. The smooth fibre of $w_\xi$ is a Weinstein 2-manifold $\Sigma_\xi$. 
 Say the boundary divisor $D_\xi$ has components indexed by rays $\nn_1, \ldots, \nn_k$ in the toric fan for $\bar{Y}_{\xi}$. 
Then  $\Sigma_\xi$ is a $k$ punctured elliptic curve, with the $k$ meridians (between successive punctures) also naturally indexed by the rays of the fan.
For notational simplicity, set $\bar{D}_i := (\bar{D}_\xi)_{\nn_i}$ and $m_i := m_{\nn_i}$. The triangulated category $D (Y_\xi)$ has the following full exceptional collection :
\begin{multline}\label{eq:exceptional-collection}
\cO_{\Gamma_{km_k}}(\Gamma_{km_k}), \cdots, \cO_{\Gamma_{k1}}(\Gamma_{k1}), \cdots, \cO_{\Gamma_{1m_1}}(\Gamma_{1m_1}),\cdots, \cO_{\Gamma_{11}}(\Gamma_{11}), \cO,
p^\ast \cO(\bar{D}_1), \\ \cdots, p^\ast\cO(\bar{D}_1+ \cdots + \bar{D}_{k-1})
\end{multline}
where $\Gamma_{ij}$ is the pullback of the $j$th exceptional curve over $\bar{D}_i$, for $i=1, \ldots, k$, $j=1, \ldots, m_i$ and $p$ denotes the blow-down map. 

On the mirror side, these correspond to a distinguished collection of vanishing cycles for $w_\xi: M_\xi \to \bC$. This is an ordered collection of exact $S^1$s on $\Sigma_\xi$, given by:
\begin{equation}\label{eq:vanishing-cycles}
\{ W_{ij} \}_{i=1, \ldots, k, j=1, \ldots, m_i}, V_0, \ldots, V_{k-1}
\end{equation}
where for fixed $i$, $W_{ij}$ is a copy of the $i$th meridian, and $V_\ell$ is given by starting with a reference longitude $V_0$ and applying $\bar{D}_i \cdot (\bar{D}_1 + \ldots + \bar{D}_{\ell-1})$ twists in the $i$th meridien. 

\subsubsection{Almost-toric fibration $\pi_\xi$.}\label{sec:almost-toric-fibration-initial}
Thirdly, $M_\xi$ is the total space of an almost-toric fibration determined by $\xi$, say $\pi_\xi: M_\xi \to B_\xi$. 
Let $\bR^2_\xi$ denote the following integral affine structure on a topological $\bR^2$.  Take the singular points to be precisely the set $S_\xi$. 
If a singular point is a multiple of $( 0, 1)^T$, take the local integral affine monodromy to be 
\begin{equation}\label{eq:monodromy-matrix} \begin{pmatrix}1  & 0 \\ 1 &1 \end{pmatrix}\end{equation} 
(no shearing: the invariant direction goes through the origin), and for general $\nn$, its conjugate by an element of $\SL_2(\bZ)$ mapping $\nn$ to $(0, 1)^T$. 
Now take $B_\xi \subset \bR^2_\xi$ to be a convex closed subset containing all of the points of $S_\xi$, and whose boundary is smooth with respect to the integral affine structure.  
The total space of an almost-toric fibration with integral affine base $B_\xi$ is determined up to symplectomorphism \cite[Corollary 5.4]{Symington}. In particular, the fibration has $m_\nn$ nodal fibres with invariant direction $\nn$, all colinear, and no other singular fibres.

The following explicit model  will be useful: start with $T^\ast T^2 \simeq T^\ast \bR^2 / \bZ^2 \simeq \bR^2 \times T^2$, and for each $\vv \in S_\xi$, `cut open' the $\bR^2$ base in invariant direction $\vv$ until reaching $\vv$, glue back using the pullback action of the relevant $\SL_2(\bZ)$ conjugate of $ \begin{pmatrix}1  & 0 \\ 1 &1 \end{pmatrix}$
and then fill back in a nodal fibre. (The Ng\d{o}c invariants of the nodal fibres do not matter; see \cite{Evans} for more background.) We have the following symplectic features:

\begin{itemize}

\item 
Assume that   $q_i$ are coordinates on $T^2$ and $p_i$ their dual coordinates. Then the one-form $\sum_i p_i dq_i$ is invariant under pull-back by $\SL_2(\bZ)$ matrices, and, away from singular fibres, descends to a well-defined primitive form for the symplectic form on $M_\xi$. In particular, we can take this to be the Liouville form near the boundary of $M_\xi$. 

\item
The fibration $\pi_\xi$ has a preferred Lagrangian section, given by taking the image of $T^\ast_{\{ 0 \}} T^2$, as $\{ 0 \} \in T^2 = \bR^2 / \bZ^2$ is fixed by all of the $\SL_2(\bZ)$ maps used to introduce branch cuts. 

\item There is a preferred Lagrangian disc above the segment  $\R_{\geq m_{\nn}} \nn \cap B_\xi$, the `visible Lagrangian' of Symington \cite{Symington}. In the Weinstein handlebody model, it corresponds to the Lagrangian co-core of the $m_\mathbf{n}$ th Weinstein 2-handle attachment in that direction. Similarly, there is a visible Lagrangian sphere above the segment from $\vv = i \nn$ to $(i+1) \nn$, say $S_\vv$, for all $i = 1, \ldots, m_\nn-1$ (the union of the co-core and core of consecutive handle attachments in the same direction). 
\end{itemize}

As long as the closed set $B_\xi$ is convex and large enough to contain all points of $S_\xi$, its precise choice will not matter: completing the Liouville domain by gluing on half-infinite ends, we get a Liouville manifold which is immediately identified with the total space of the almost-toric fibration with base the whole of $\bR_\xi^2$ (and in particular, we readily have preferred identifications between any such completions).

The spheres $S_\vv$ are also distinguished from the perspective of mirror symmetry. We will use the following:

\begin{lemma}\label{lem:mirrors-of-(-2)-curves-blowup} \cite[Proposition 5.2]{HK2}
Let $\xi \in \CY$, and suppose that $m_\nn \geq 2$ for some ray $\nn$ of  $\xi$. Let $C_{\vv} \subset Y_\xi$ denote the $(-2)$ curve which is the  strict transform of the exceptional curve introduced by the blow-up indexed by $\vv = i \nn$, some $ 1 \leq i \leq m_\nn-1$. Then under the equivalences of Theorem \ref{thm:original-hms}, $i_\ast \cO_{C_{\vv}} (-1)$ is mirror to the visible Lagrangian sphere $S_\vv$. For other values of $k \in \bZ$, the mirror to $i_\ast \cO_{C_{\vv}} (-k)$ is also a Lagrangian two-sphere, fibred over the same segment in $B_\xi$ as $S_\vv$, and can be obtained by applying a Lagrangian translation to $S_\vv$. 
\end{lemma}

(Recall that Lagrangian translations are a class of symplectomorphisms introduced in \cite{HK2}.)

 \subsubsection{Describing $\cF^\to (w_\xi)$ in terms of $\pi_\xi$}

In the HMS paper \cite{HK1}, results about Fukaya categories are essentially all done in terms of the Lefschetz fibration. 
It will sometimes be better to work with the almost-toric fibration $\pi_\xi: M_\xi \to B_\xi$ instead, following  \cite{Sylvan, GPS1, GPS2}.

\begin{definition}
Let $\mathfrak{f}_\xi \subset \partial M_\xi$ be given by the union of the following Legendrians:
\begin{itemize}
\item the boundary of the distinguished Lagrangian section of $\pi_\xi$ inherited from $T^\ast_{ \{ 0 \} } T^2$ (described in the second bullet point in Section \ref{sec:almost-toric-fibration-initial});
\item the boundaries of each of the visible Lagrangian thimbles living over the branch cuts $\R_{ \geq m_\nn } \nn \cap B_\xi$  (described in the third bullet point in Section \ref{sec:almost-toric-fibration-initial}). 
\end{itemize}
\end{definition}

The subset $\mathfrak{f}_\xi \subset \partial M_\xi$ is a mostly Legendrian stop in the sense of  \cite[Definition 1.7]{GPS2}. Following \cite[Section 2]{GPS2}, let $\W(M_\xi, \frak{f}_\xi)$ denote the wrapped Fukaya category of $M_\xi$ stopped at $\mathfrak{f}_\xi$.  (We work with Liouville manifolds with stops rather than Liouville sectors.)
\begin{lemma} \label{lem:directed-category-iso-stopped-category}
There is an isomorphism of $A_\infty$ categories:
$$ \cF^{\to} (w_\xi) \simeq \W(M_\xi, \frak{f}_\xi)
$$
\end{lemma}

\begin{proof}
By \cite[Corollary 1.17 and Proposition 8.20]{GPS2}, we have an $A_\infty$ quasi-isomorphism between  $\cF^{\to} (w_\xi) $ and $ \W(M_\xi, \frak{f})$, where $\frak{f}$ is the Weinstein skeleton of a copy of $\Sigma_\xi$ `at infinity'.
We want to identify  $\frak{f}$ with  $\frak{f}_\xi$. 

We inspect the identification set up in \cite[Sections 5 and 6]{HK1} to  describe $\frak{f}$ in terms of the almost-toric fibration. Fix a reference fibre near infinity in the Lefschetz fibration, say  $\Sigma_{\xi, \star}$.
We can choose the Liouville forms so that the Lagrangian skeleton of  $\Sigma_{\xi, \star}$ is, in the notation of \cite[Definition 3.1]{HK1}, the union of the reference longitude $V_0$ and all of the meridians $W_\mathbf{n}$, where $\mathbf{n}$ varies over all rays of $\xi$. 
These are the boundaries of well-understood Lagrangian thimbles for $w_\xi$. In particular, we see that under our identification with the Weinstein handlebody description of $M_\xi$, $V_0$ is the boundary of a cotangent fibre of $T^2$ \cite[Section 5.2]{HK1}, and each $W_\mathbf{n}$ the boundary of the co-core of the $m_\mathbf{n}$th (i.e.~final) handle attachment for direction $\mathbf{n}$ \cite[Section 6.2]{HK1}.
In turn, in terms of the almost-toric fibration $\pi_\xi$, we get from Section \ref{sec:almost-toric-fibration-initial} that $\frak{f}$ lies over the boundary of $B_\xi$ and is given by the union of the restrictions of the distinguished Lagrangian section inherited from $T^\ast_{\{0\}} T^2$, which corresponds to $V_0$; and of each of the visible Lagrangian thimbles living over branch cuts, which correspond to the $W_{\mathbf{n}}$. This completes the proof.
\end{proof}

This is independent of our choices for the almost-toric model: suppose $B_\xi' \subset \bR^2_\xi$ is a second convex subset containing all of $S_\xi$, with smooth boundary. Let $M_\xi' = \pi_\xi^{-1} (B_\xi')$, and let $\frak{f}_\xi'$ be the restriction of the same Lagrangian discs as above to $\pi_\xi^{-1} (\partial B_\xi')$. We later repeatedly use the following:

\begin{lemma}\label{lem:change-ATF-base-stop}
With the notation as above, there are natural $A_\infty$ quasi-isomorphisms
$$
\cW(M_\xi, \frak{f}_\xi) \simeq \cW(M_\xi', \frak{f}_\xi') \qquad \text{and} \qquad \cW(M_\xi) \simeq \cW(M_\xi')
$$
\end{lemma}

\begin{proof}
This follows from the definitions in \cite{Abouzaid-Seidel} for the case with no stop. In the stopped case, recall that $\cW(M_\xi, \frak{f}_\xi)$ is defined using the completed Liouville manifold $M_\xi^+= M_\xi \cup \big( \partial M_\xi \times [0,\infty) \big)$, with stop $\frak{f}_\xi \times \{ \infty \} \subset \partial_\infty M_\xi^+$, the formal boundary at infinity; see \cite[Section 1.1]{GPS2}. As observed at the end of Section \ref{sec:almost-toric-fibration-initial}, both completions are the total spaces of the `full' almost-toric fibrations over $\bR^2_\xi$, with the same Liouville forms, and identical stops. 
\end{proof}

\subsection{Mirror constructions and directed systems} \label{sec:mirror-system-compatibility}

Suppose that $\xi, \xi' \in \CY$ with $\xi' \geq \xi$.  Then there are the following compatibilities between the mirror constructions.

\subsubsection{Weinstein handlebody.}
First, as a Weinstein handlebody, $M_{\xi'}$ is obtained by attaching $m_\nn'-m_\nn$ further Weinstein two-handles to $M_{\xi}$, for all rays in the fan for $(\oY_{\xi'}, \oD_{\xi'})$. This induces an inclusion of Weinstein domains $M_\xi \hookrightarrow M_{\xi'}$  (up to Weinstein homotopy, as in \cite[Section 11.6]{Cieliebak-Eliashberg}).

\subsubsection{Lefschetz fibration.}
Fix $\xi \in \CY$. Recall that the explicit toric model for  $(Y_\xi,D_\xi)$ determines a Lefschetz fibration $w_\xi$ with total space  $M_\xi$ and fibre $\Sigma_\xi$.

Now assume $\xi' \geq \xi \in \CY$. We have a natural inclusion of $i: \Sigma_\xi \hookrightarrow \Sigma_{\xi'}$, given by a handle attachment \cite[Proposition 3.11]{HK2}. The distinguished collection of vanishing cycles for $w_{\xi'}$ from Equation \ref{eq:vanishing-cycles} can be mutated to another collection, say $V_{\xi', 1}, \ldots, V_{\xi', l_{\xi'}}$,
 such that the final $l_{\xi}$ ones are the images under $i$ of a distinguished collection for $w_\xi$, say, $V_{\xi, 1}, \ldots, V_{\xi, l_\xi}$. See \cite[Section 3.2]{HK1}. (We will revisit this shortly when discussing the HMS compatibilities of inclusions.)

 These two inclusions induce an inclusion of total spaces $M_\xi \hookrightarrow M_{\xi'}$, which naturally agrees (up to Weinstein homotopy) with the inclusion of Weinstein handlebodies above. Recall that if there are only corner blow-ups, the two Lefschetz fibrations differ by a sequence of stabilisations, and the total spaces are naturally Weinstein deformation equivalent.

\subsubsection{Almost toric fibrations}

Suppose you have $\xi' \geq \xi$. Up to enlarging the collar neighbourhood of the boundary of $M_{\xi'}$, we can arrange for there to be an inclusion of integral affine manifolds   $B_\xi \subset B_{\xi'}$: $\pi_{\xi'}$ is given by starting with $\pi_\xi$, enlarging the base from $B_\xi$ to $B_{\xi'}$ and adding $m'_{\mathbf{n}} - m_\mathbf{n}$ nodal fibres with invariant direction $\mathbf{n}$ in $B_{\xi'} \backslash B_\xi$. The corresponding inclusion $M_{\xi} \subset M_{\xi'}$ is compactible with the previous ones.

\subsubsection*{Maps on Fukaya categories}
Recall that gradings for the Fukaya category depend on a choice of trivialisation of $(\Lambda^2 T^\ast M_\xi)^{\otimes 2}$. These choices form an affine space over $H^1(M_\xi; \bZ)$, which vanishes as soon as $m_{\nn_i} >0$ for two linearly independent $\nn_i$. When $H^1(M_\xi; \bZ) \neq 0$, we use the unique trivialisation given by restricting the trivialisation for an arbitrary $M_\eta$ with $\eta \geq \xi$. 

Now suppose $\xi' \geq \xi$. 
From the discussion above on compatibilities of structures, this induces preferred $A_\infty$ functors on the Floer-theoretic structures of interest:
\begin{itemize}

\item
 $\io_\ast: \cF (\Sigma_\xi ) \to \cF (\Sigma_{\xi'} )$ induced by  $\Sigma_\xi \hookrightarrow \Sigma_{\xi'}$;

\item $\io_\ast: \cF^{\to} (w_\xi) \to \cF^{\to} (w_{\xi'})$ induced by the inclusion of Lefschetz fibrations;

\item $\io^\ast: \cW(M_{\xi'}) \to \cW(M_{\xi})$, the Viterbo restriction from $M_{\xi'}$ to $M_\xi$. 
\end{itemize}

The first two are fully faithful: this follows from definitions plus Abouzaid's integrated maximum principle \cite[Lemma 7.5]{Seidel-book}.

\begin{remark}
Working with stops, the $A_\infty$ functor $\io_\ast: \cF^{\to} (w_\xi) \to \cF^{\to} (w_{\xi'})$ is equivalent to an $A_\infty$ functor $\io_\ast: \cW(M_\xi, \frak{f}_\xi) \to  \cW(M_{\xi'}, \frak{f}_{\xi'})$, though this isn't entirely straighforward to set up within the Liouville sector formalism of \cite{GPS1, GPS2}. 
\end{remark}

We next check that the HMS equivalences are compatible with the partial order on $\CY$.

\begin{proposition} \label{prop:hms-inclusion-compatibility}
Suppose $\xi' \geq \xi \in \CY$. 
Then the following diagrams commute up to $A_\infty$ homotopy:
$$
 \xymatrix{
\perf D_\xi \ar[r]^-\simeq  \ar[d]_{p^\ast} &  \cF(\Sigma_{\xi}) \ar[d]_{\io_\ast} \\
\perf D_{\xi'} \ar[r]^-\simeq  & \cF(\Sigma_{\xi'})  \\
}
\qquad 
\xymatrix{
\coh Y_{\xi} \ar[r]^-\simeq   \ar[d]_{p^\ast} &   \cF^\to (w_{\xi}) \ar[d]_{\io_\ast} \\
\coh Y_{\xi'}  \ar[r]^-\simeq  &   \cF^\to (w_{\xi'}) \\
}
\qquad
\xymatrix{
\coh U_{\xi} \ar[r]^-\simeq    &  \cW(M_{\xi}) \\
\coh U_{\xi'}  \ar[r]^-\simeq  \ar[u]_{i^\ast} & \cW(M_{\xi'})   \ar[u]^{\io^\ast} \\
}
$$
where $p: (Y_{\xi'}, D_{\xi'}) \to  (Y_{\xi}, D_{\xi}) $ is the blow-down map, and $i: U_\xi \to U_{\xi'}$ the inclusion.

\end{proposition}

\begin{proof} By \cite[Theorems 1.2 and 1.5]{Genovese} (building on \cite{Toen}), the functor of triangulated categories
 $p^\ast: D (Y_\xi) \to  D (Y_{\xi'})$ has a unique dg lift to a dg functor $\coh Y_\xi \to \coh Y_{\xi'}$, up to dg homotopy. (To use this we just need, firstly, that the pair $(Y_\xi, Y_{\xi'})$ satisfies the hypotheses of \cite[Theorem 1.2]{Genovese}: they are separated and Noetherian, their product is Noetherian, and that for both of them any perfect complex is strictly perfect; and secondly, that $p^\ast$ is fully faithful.)  Similarly for   $p^\ast:  \Perf D_\xi \to \Perf D_{\xi'}$ and $i^\ast: D (U_{\xi'}) \to D (U_\xi)$. Thus to prove the claim, it is enough to show that the induced diagrams of homotopy diagrams commute.

Say again that the rays in the fan for $\xi$ are $\nn_1, \ldots, \nn_k$. We again use the notation $D_i = (D_\xi)_{\nn_i}$, $\bar{D}_i = (\bar{D}_\xi)_{\nn_i}$, and $m_i = m_{\nn_i}$. 
For each of the three homotopy-category diagrams, it's enough to check that the images of the generators in our explicit collections agree, as all of the vertical functors are fully faithful. 

\emph{First diagram.}
$\Perf(D_\xi)$ is generated by $\cO$ and  $i_\ast \cO_{x_i}$, $i=1, \ldots, k$, where $x_i$ is an explicit point on $D_i \backslash \sqcup_{j \neq i} D_j $. (Using the identification with $ \bC^\ast$ coming from our explicit choice of toric model, $x_i$ corresponds to $-1$.) 
These are mirror to explicit Lagrangian $S^1$s on $\Sigma_\xi$: a preferred longitude $V_0$, and $k$ meridians. 
Under $p^\ast$, $\cO_{D_\xi}$ is pulled back to $\cO_{D_{\xi'}}$, and $i_\ast \cO_{x_i}$ to $i_\ast \cO_{x_{i'}}$ (with the obvious relabeling of components). This is then immediately compatible with $\Sigma_\xi \hookrightarrow \Sigma_{\xi'}$.

\emph{Second diagram.} Recall we have two mirror full exceptional collections: for 
$D(Y_\xi)$, the one given by 
\begin{multline*}
\cO_{\Gamma_{km_k}}(\Gamma_{km_k}), \cdots, \cO_{\Gamma_{k1}}(\Gamma_{k1}), \cdots, \cO_{\Gamma_{1m_1}}(\Gamma_{1m_1}),\cdots, \cO_{\Gamma_{11}}(\Gamma_{11}), \cO,
p^\ast \cO(\bar{D}_1), \\ \cdots, p^\ast\cO(\bar{D}_1+ \cdots + \bar{D}_{k-1})
\end{multline*}
where $\Gamma_{ij}$ is the pullback of the $j$th exceptional curve over $\bar{D}_i$, for $i=1, \ldots, k$, $j=1, \ldots, m_i$; and on the mirror side, the distinguished collection of vanishing cycles for $w_\xi: M_\xi \to \bC$, 
$$
\{ W_{ij} \}_{i=1, \ldots, k, j=1, \ldots, m_i}, V_0, \ldots, V_{k-1} \subset \Sigma_\xi
$$
where for fixed $i$, $W_{ij}$ is a copy of the $i$th meridian, and $V_\ell$ is given by starting with our prefered longitude $V_0$ and applying $\bar{D}_i \cdot (\bar{D}_1 + \ldots + \bar{D}_{\ell-1})$ twists in the $i$th meridien. 

Assume first that  $\xi' > \xi$ is given by interior blow-ups. Then we immediately get that the second diagram commutes: the two vertical maps are mirror fully faithful inclusions of exceptional collections, with the bottom isomorphism given by iteratively adding mirror exceptional objects to the left of each of the collections. 

Now assume that $\xi' > \xi$  also involves corner blow-ups. Then we first mutate the exceptional collections for $\xi'$, to `pull out' the generators associated to new components of the boundary $\bar{D}_{\xi'}$. 
This can be done explicitly, see \cite[Prop. 3.11 and 3.29]{HK1}.  
We get a full exceptional collection for $\coh Y_{\xi'}$, say 
$\cE_{{\xi'},1}, \ldots, \cE_{\xi',{l_{\xi'}}}$, 
such that the subcollection
$\cE_{\xi',{l_{\xi'} - l_\xi +1}}, \ldots, \cE_{\xi',{l_{\xi'}}}$ is the pullback of a full exceptional collection for $\coh Y_{\xi}$, and $\cE_{{\xi'},1}, \ldots,\cE_{{\xi'}, l_{\xi'} - l_\xi}$ are exceptionals introduced by the blow-ups using Orlov's formula \cite{Orlov}. 
Performing the same sequence of mutations on the mirror side,  the resulting  distinguished collection of vanishing cycles is  $V_{{\xi'},1}, \ldots, V_{{\xi'},l_\xi'}$, as introduced above. (Say the associated thimbles are $\vartheta_{{\xi'},1}, \ldots, \vartheta_{{\xi'},l_\xi'}$.)
The upshot is that on both sides, we can go from the mirror exceptional collections   for $\xi$ to  ones for $\xi'$ by adding one exceptional object for each blow-up to the left of the list. In particular, the second diagram commutes, with fully faithful vertical maps.

\emph{Third diagram.}
From the proof of \cite[Theorem 4.9]{HK1}, we have a commutative diagram
$$
\xymatrix{
\ccoh U_{\xi} \ar[r]^-\simeq    &  \cW(M_{\xi}) \\
\ccoh Y_{\xi}  \ar[r]^-\simeq  \ar[u]_{i^\ast} & \cF^\to (w_\xi)  \ar[u]^{\rho} \\
}
$$
where $i: U_\xi \to Y_\xi$ is the inclusion, and $ \rho : \cF^\to (w_\xi) \to    \cW(M_{\xi})$ is a localisation functor. Using the quasi-isomorphism of Lemma \ref{lem:directed-category-iso-stopped-category}, $\rho$ is given by $\cF^\to (w_\xi) \simeq \cW(M_\xi, f_\xi) \to \cW(M_\xi)$, where the map $ \cW(M_\xi, f_\xi) \to \cW(M_\xi)$ is the stop-removal procedure introduced in \cite[Section 1.6]{GPS2}.
 These claims can be found in \cite[Section 4.5]{HK1}. (The reader might want to note that the key input is \cite[Theorem 4.7]{HK1}, which shows that  the $A_\infty$-bimodule structures given by $(\cF^\to (w_\xi), \cF(\Sigma_\xi))$ and $(\ccoh Y_\xi, \pperf D_\xi)$ are equivalent.) 
 Generators for $\ccoh U_\xi$ are given by restricting those for $\ccoh Y_\xi$. On the mirror side, $ \cW(M_\xi)$ is generated by the Lagrangian thimbles associated to the distinguished collection of vanishing cycles for $w_\xi$, now viewed as objects of the unstopped wrapped category. 

Now consider $\xi' \geq \xi$. Using the same notation as for the case of the second diagram, generators for  $ \cW(M_{\xi'})$  are given by taking the collection of Lagrangian thimbles $\vartheta_{{\xi'},1}, \ldots, \vartheta_{{\xi'},l_{\xi'}}$ for $w_{\xi'}$ and viewing them as elements of $ \cW(M_{\xi'})$. 
Similarly, $\ccoh U_{\xi'}$ is generated by the restrictions of $\cE_{{\xi'},1}, \ldots,\cE_{{\xi}, l_{\xi'}}$. From the discussion of the second diagram, $\cE_{{\xi'},1}, \ldots,\cE_{{\xi'}, l_{\xi'} - l_\xi}$ pull back to the zero object in $\ccoh U_{\xi}$, and $\cE_{\xi',{l_{\xi'} - l_\xi +1}}, \ldots, \cE_{\xi',{l_{\xi'}}}$ pull back to generators for $\coh U_{\xi}$. Similarly for the $\vartheta_{{\xi'},1}, \ldots, \vartheta_{{\xi'},l_{\xi'}}$ on the mirror side (again, see \cite[Prop. 3.11 and 3.29]{HK1}). The claim about the third diagram follows.\end{proof}

\begin{remark}\label{rem:diagreplacement}  Proposition \ref{prop:hms-inclusion-compatibility} is stated only up to $A_{\infty}$ quasi-isomorphism. In particular, we are free to choose new $A_\infty$ (in fact, dg) models for the Fukaya categories $\cF(\Sigma_\xi)$, $\cF^{\to} (w_\xi)$ and $\cW(M_\xi)$, by setting them to be strictly equal to their B-side side counterparts $\perf D_\xi$, $\coh Y_\xi$ and $\coh U_\xi$ (and for $\xi' > \xi$, using the B-side models for pullback/pushforward maps). Unless otherwise specified, we now take this viewpoint, which will be technically convenient for defining (homotopy) limits. 
\end{remark}


\subsection{Universal objects: $A$ side and homological mirror symmetry}\label{sec:hms_universal}

\begin{definition} We define the universal cusp Milnor fibre, $M_\un$, to be the direct limit of the Weinstein handlebodies $M_\xi$ under our system of inclusions:
$$M_\un = \varinjlim_{\xi \in \CY} M_\xi.
$$
Similarly, we define $\Sigma_\univ$ to be the direct limit 
$$
\Sigma_\un = \varinjlim_{\xi \in \CY} \Sigma_\xi.
$$
\end{definition}

Note that $M_\un$ and $\Sigma_\un$ are well-defined Weinstein manifolds, albeit of infinite rather than finite type. 
From Section \ref{sec:mirror-system-compatibility}, $M_\un$ inherits two additional structures. 
First,  $M_\un$ is naturally the total space of an almost-toric fibration, with countably infinitely many nodal fibres. One explicit model for it is to take $\pi_\un: M_\un \to \bR_\univ^2$, where $\bR^2_\univ$ is the integral affine structure on $\bR^2$ such that we have nodal fibres over each point of $\bZ^2 \backslash \{ 0\}$ and invariant direction through the origin (with branch cut the half-line starting at the critical point which does not go through the origin). 
Second, $M_\un$ admits a map $w_\un: M_\un \to \bC$, naturally generalising a Lefschetz fibration: it has countably infinitely many critical points, all of which are of complex Morse type, and smooth fibres  $\Sigma_\un$.

\begin{definition} \label{def:asidecats}Define the Fukaya categories $\cW(M_{\xi})$, $\cF (\Sigma_\xi)$, and $\cF^\to (w_\xi)$ according to Remark \ref{rem:diagreplacement} to ensure that our diagrams are strictly commutative, so we can take their limits and homotopy colimits as we did in Section \ref{sec:bsidesystems}. 
The  wrapped Fukaya category of $M_\un$ is 

$$\cW(M_{\un}):= \varprojlim_{\xi \in \CY} \cW(M_\xi).$$

The Fukaya category of $\Sigma_\un$ is 

$$
 \cF(\Sigma_{\un}):=  \holim_{\xi \in \CY} \cF (\Sigma_\xi).
$$

The directed Fukaya category of $w_\un$ is 

$$
\cF^\to (w_{\un}):= \holim_{\xi \in \CY}  \cF^\to (w_\xi).
$$

\end{definition}

The infinite type cases aren't considered in standard set-ups for Fukaya categories (for instance \cite{Seidel-book, GPS1}), which is why we work instead with limits.  
In the case of the wrapped category, this parallels the definition in \cite{Seidel_biased} of the symplectic cohomology of Weinstein manifolds of infinite type. 

The following is now immediate from the definitions together with Proposition \ref{prop:hms-inclusion-compatibility}. 

\begin{corollary} 	   \label{cor:universal-exact-hms}
We have a triple of $A_\infty$ quasi-isomorphisms: 
$$
\coh Y_\un \simeq   \cF^\to (w_\un) \qquad
\perf U_{\un}  \simeq    \cW(M_{\un}) \qquad
\perf D_{\un} \simeq   \cF(\Sigma_{\un})
$$
For any $\xi \in \CY$, these are compatible with the obvious inclusions of, or restrictions to, the finite-level isomorphisms.

\end{corollary}

\section{The mirror to $\Bir_e(\bP^2, \Omega)$} \label{sect:symplectomorphisms}

\subsection{Birational transformations of $\bP^2$}
Let \[ \Omega =d \log x \wedge d \log y  \] be the standard meromorphic volume form on $\bP^2$ with simple poles along the components of the toric boundary divisor. We define $\Bir(\bP^2, \Omega)$ as the group of birational transformations of $\bP^2$ which preserve $\Omega$ on their domain of definition, and 
$\Bir(\bP^2, \pm\Omega)$ as the group of transformations which preserve $\Omega$ up to sign. (Elements of $\Bir(\bP^2, \Omega)$ are often refered to as `symplectic birational transformations', though we will not use this terminology to avoid confusion.) Blanc \cite{Blanc} proved that 
$$ \Bir(\bP^2, \Omega) = \langle \SL_2(\bZ), E, (\bC^\ast)^2 \rangle$$
where 
 the complex torus acts by left multiplication: 
 \[(x,y) \mapsto (\lambda_1 x, \lambda_2 y);\]
  a matrix 
\[ A = \begin{pmatrix} a & b \\ c & d \end{pmatrix} \in \SL_2(\bZ) \]
acts on  $(\bC^\ast)^2$ by 
\[ (x,y) \mapsto  (x^a y^c, x^b y^d);\]
 and $E$ acts by 
\[ (x,y) \mapsto (x, y (x+1)\inv).\] 
To generate the group $\Bir(\bP^2, \pm\Omega)$, we also include the element 
\[ \begin{pmatrix} -1 & 0 \\ 0 & 1 \end{pmatrix} \in \GL_2(\bZ) \]
acting by $(x,y) \mapsto (x\inv, y)$. We also introduce notation for the $\SL_2(\bZ)$ conjugates of $E$: given any $\nn \in (\Z^2)_{\mathrm{prim}}$, we define
$ E_{\nn} = \phi_A\inv E \phi_{A}$
for any $A \in \SL_2(\Z)$ with $A \nn=(0,1)$. These are often called `elementary transformations'. 

\begin{definition}
Let $\Bir_e(\bP^2, \Omega)$, respectively $\Bir_e(\bP^2, \pm \Omega)$, be the subgroup of $\Bir(\bP^2, \pm \Omega)$ generated by $E$ and $\SL_2(\bZ)$, respectively $\GL_2(\bZ)$.  
\end{definition}

For a pair $\xi_1, \xi_2 \in \CY$, the embeddings $\iota_{\xi_i}: (\C^*)^2 \into Y_{\xi}$ determined by the explicit toric model of $Y_{\xi}$ define a canonical map
\[ \Bir(Y_{\xi},Y_{\eta}) \to \Bir(\bP^2);\]
 moreover, it makes sense to say that a given element of $\Bir(\bP^2)$ is regular on a surface $Y_{\xi}$ or $U_{\xi}$. 

\begin{example}\label{ex:SL2Zaut} Let $A \in \GL_2(\Z)$ and $\xi \in \CY$, and let $A_*{\xi}$ index the surface where we have modified our explicit toric model for $Y_{\xi}$ by precomposing with the action of $A$; it is clear that $A$ extends to a biholomorphism $Y_{\xi} \to Y_{A_*\xi}.$
See Figure \ref{fig:Aaut} for an example. \end{example} 

\begin{example} \label{ex:Eaut} The transformation $E$ extends to a biholomorphism between the blowup of $\bP^1 \times \bP^1 $ at the distinguished point on $D_{(0,1)}$  and the blowup of the first Hirzebruch surface at the distinguished point on $D_{(0,-1)}$. Both $\bP^1 \times \bP^1$ and $\bF_1$ admit toric morphisms to $\bP^1$; the birational map $\bP^1 \times \bP^1 \dashrightarrow \bF_1$ factors as the blowup at a point in the fibre over $[-1:1] \in \bP^1$ and then the blowdown of the proper transform of the fibre itself. See Figure \ref{fig:Eaut}. \end{example}

\begin{figure} \label{fig:Aaut}

\begin{tikzpicture}
\draw[thick, ->] (0,0) -- (1.5,0) node[right] {};
\draw[thick, ->] (0,0) -- (0,1.5) node[above] {};
\draw[thick, ->] (0,0) -- (-1.5,-1.5) node[below] {};

\draw[-> ,thick] (2.2, 0) arc (260 : 280 : 4);
\node at (3,.75) {$\begin{pmatrix}-1 & 0 \\ 0 & -1 \end{pmatrix}$};

\begin{scope}[xshift=6cm, yshift=0];
\draw[thick, ->] (0,0) -- (-1.5,0) node[right] {};
\draw[thick, ->] (0,0) -- (0,-1.5) node[above] {};
\draw[thick, ->] (0,0) -- (1.5,1.5) node[below] {};
\end{scope}
\end{tikzpicture}
\caption{The birational transformation $(x,y) \mapsto (x\inv, y\inv)$ extends to a biholomorphism between two different toric compactifications of $(\C^*)^2$ to $\bP^2$s pictured above.} 
\end{figure} 

\begin{figure}
\begin{tikzpicture}\label{fig:Eaut}
\draw[thick, ->] (0,0) -- (1.5,0) node[right] {};
\draw[thick, ->] (0,0) -- (0,1.5) node[above] {};
\draw[thick, ->] (0,0) -- (-1.5,0) node[below] {};
\draw[thick, ->] (0,0) -- (0,-1.5) node[right] {};
\node at (0,0) {$\cdot$};
\node at (0,.75) {$\times$};
\draw[-> ,thick] (2.2, 0) arc (260 : 280 : 4);
\node at (3,.75) {$E$};

\begin{scope}[xshift=6cm, yshift=0];
\draw[thick, ->] (0,0) -- (1.5,0) node[right] {};
\draw[thick, ->] (0,0) -- (0,1.5) node[above] {};
\draw[thick, ->] (0,0) -- (-1.5,1.5) node[below] {};
\draw[thick, ->] (0,0) -- (0,-1.5) node[right] {};
\node at (0,0) {$\cdot$};
\node at (0,-.75) {$\times$};
\end{scope}
\end{tikzpicture}
\caption{The cluster transformation $E: (x,y) \mapsto (x, y (1+x)\inv)$ defines a regular map $\Bl_{(-1, \infty)} \bP^1 \times \bP^1 \to \Bl_{(-1, 0)} \bF_1$.}  
\end{figure}

\begin{prop}\label{prop:extendphi} If $\phi \in \Bir_{e}(\bP^2, \pm \Omega)$ defines a biholomorphism  $Y_{\xi} \to Y_{\phi_* \xi}$, then $\phi$ restricts to biholomorphisms $\phi: U_{\xi} \to  U_{\phi_*\xi}$ and 
		$\phi: D_{\xi} \to D_{\phi_*\xi}$. 
Moreover, for any $\eta > \xi $, there exists an element $\phi_* \eta \in \CY$ such that the following diagram commutes: 
\[ 
\begin{tikzcd}
Y_{\eta} \arrow[r, "\phi"] \arrow[d] & Y_{\phi_* \eta} \arrow[d]    \\
Y_{\xi}           \arrow[r, "\phi"]                    & Y_{\phi_* \xi}.
\end{tikzcd}
\]

\end{prop}

\begin{proof} Since the boundary divisors $D_{\xi}$ and $D_{\phi_*\xi}$ are each characterised as the set of points in $Y_{\xi}$ or $Y_{\phi_*\xi}$ (respectively) on which $\Omega$ has a simple pole, any biholomorphism between $Y_{\xi}$ and $Y_{\phi_*{\xi}}$ preserving $\Omega$ will necessarily restrict to a biholomorphism between the two boundary components and a biholomorphism between the two interiors as well. 

Now let $\phi \in \Bir_{e}(\bP^2, \Omega)$ define a biholomorphism $Y_{\xi} \to Y_{\phi_*\xi}$ and choose any $\eta > \xi$. It suffices to consider the case when $\phi$ is one of our distinguished generators of $\Bir_{e}(\bP^2, \Omega)$ and $Y_{\eta}$ is constructed by blowing up either a corner point or a distinguished interior point on $D_{\xi}$. Call this point $p$, and note that $\phi$ automatically extends to a biholomorphism $Y_{\eta} \to \mathrm{Bl}_{\phi(p)} Y_{\phi_*{\xi}}$. If $\phi(p)$ is itself a distinguished point on $D$, then $\mathrm{Bl}_{\phi(p)} Y_{\phi_*{\xi}}$ (along with the explicit toric model factoring through that of $Y_{\xi}$) is an element of $\CY$. This will evidently be the case if $p$ is either a corner point or if $p$ is a distinguished interior point and $\phi \in \GL_2(\Z)$. So it suffices to check that $E$ sends the distinguished interior points to distinguished interior points. 

In coordinates we have that for any $(n_1, n_2) \in (\Z^2)_{\prim}$,
\[  \lim_{t \to 0} \, E( \lambda_1 t^{n_1}, \lambda_2 t^{n_2} ) = 	\begin{cases}	 \lim_{t \to 0} \,( \lambda_1 t^{n_1}, \lambda_2 t^{n_2}) &n_1 > 0 \text{ or } (n_1,n_2) = (0,1) \text{\, and } \lambda_1\neq 1 \\
													  \lim_{t \to 0} \,( \lambda_1 t^{n_1}, \lambda_1 \lambda_2 t^{n_1+n_2}) &n_1 <  0 \\ \end{cases}\]
Comparing with Equation \eqref{eq:Dn},
we conclude firstly that $E$ must map $D_{\nn}$ to $D_{E_*( \nn)}$, where $E_*( \nn)$ is the image of $\nn \in \zm$ under the piecewise-linear transformation implicitly defined above, and secondly that the maps 
\[E|_{D_{\nn}^\mathrm{int}}: D_{\nn}^{\mathrm{int}} \to  D_{E_*(\nn)}^{\mathrm{int}} \]
preserve the natural $\C^*$ coordinate charts on the interiors of the divisor components. In particular, $E$ will always send the distinguished point in $D_{\nn}$ to the distinguished point in $D_{E_*(\nn)}$ as desired.
\end{proof} 

\begin{remark} The above proposition fails if we consider all the elements in the larger group $\Bir (\bP^2, \pm \Omega)$: for example, note that the map 
 $(x,y) \mapsto (2 x, y)$
  defines a volume-preserving automorphism of $\bP^2$ which fails to extend to other surfaces in our system. However, the following lemma shows that this is a limited phenomenon. \end{remark}

\begin{lemma}\label{lemma:compatibility}Assume that $\xi_1 \in \CY$ is such that $h^0(Y_{\xi_1}, -K_{Y_{\xi_1}})=1$ and $m_{\nn_1}, m_{\nn_2} \geq 1$ for some pair of linearly independent elements of $(\Z^2)_{\prim}$. Then any element $\phi \in \Bir(\bP^2)$ extending to a biholomorphism $\phi: Y_{\xi_1} \to Y_{\xi_2}$ must be an element of $\Bir_e(\bP^2, \pm \Omega)$. 
\end{lemma}

\begin{proof}For this proof it will be useful notationally to distinguish the underlying morphism from $Y_{\xi_1}$  to  $Y_{\xi_2}$ from the corresponding birational transformation of $\bP^2$; we denote the former by $\phi_Y$. 

 Let $\iota_{i}: (\C^*)^2 \to Y_{\xi_i}$ be the canonical inclusions.  Recall that the $\iota_i$ define which points we call distinguished on the boundary divisor $\bar{D}_{\xi_i}$; the toric model $p_{\xi_2}: Y_{\xi_2} \to \bar{Y}_{\xi_2}$, which contracts exceptional divisors to distinguished points, uniquely determines $\iota_2$ as soon as two of those points lie on components of the boundary divisor indexed by linearly independent elements of $(\Z^2)_{\prim}$ as we have assumed.

 For $i=1,2$, let $\Omega_{i}$ denote the meromorphic volume form on $Y_{\xi_i}$ satisfying
$\iota_{i}^* \Omega_{i}= \Omega.$
Recall (\cite[Section 5.1]{Looijenga} or \cite[Section 3]{Friedman}) that the $\Omega_{i}$ are normalised such that
\[ \int_{[\gamma_{i}]} \Omega_{i} = 1 \]
where $[\gamma_{i}] \in H_2(U_{\xi_i}; \Z)$ is the class of the torus which generates 
\begin{align} \label{eqn:kernel} K_{i}= \ker(H_2(U_{\xi_i};\Z) \to H_2(Y_{\xi_i}; \Z)) \cong \Z, \end{align}
where the orientation on $\gamma_{i}$ is determined by the orientation on the boundary divisors $D_{\xi_i}$. Since $h^0(Y_{\xi_1}, -K_{Y_{\xi_1}})=1$ we conclude that $\phi_{Y}^* \, \Omega_{2}= \lambda \Omega_1$ for some $\lambda \in \C^*$. The same condition also implies that each $D_{\xi_i}$ is the unique divisor representing the anticanonical class, so 
\[ \phi_Y(U_{\xi_1})=\phi_Y(Y_{\xi_1} \setminus D_{\xi_1})= Y_{\xi_2} \setminus D_{\xi_2} = U_{\xi_2},\] whence $(\phi_{Y})_*(K_1)=K_2$. Thus $\phi_{Y}^*(\Omega_{2})=\pm \Omega_{1}$. 

Assume now that $\phi_{Y}^*(\Omega_{2})= \Omega_{1}$, so $\phi_Y$ respects the orientations on the $D_{\xi_i}$. The composition $\pi_2 \circ \phi_Y: Y_{\xi_1} \to \bar{Y}_{\xi_2}$  gives us a change of toric model in the sense of \cite{HK1} (what we might call a change in \emph{inexplicit} toric model). 
By \cite[Proposition 3.7]{HK1}, the birational map $\phi_M: \bar{Y}_{\xi_1} \dashrightarrow \bar{Y}_{\xi_2}$  
can be factored as a sequence of elementary transformations: there exist embeddings $\iota_i': (\C^*)^2 \into  \bar{Y}_{\xi_2}$ such that the exceptional divisors contracted by $p_i$   map to the $\iota_i'$-distinguished points, and \[ (\iota_2')\inv \phi_Y  \iota_1' = A \circ E_{\nn_1} \circ \cdots \circ E_{\nn_k} \]
for some $\nn_1, \cdots, \nn_k \in (\Z^2)_{\mathrm{prim}}$ and $A \in \SL_2(\Z)$. But then the second condition on $Y_{\xi_1}$ implies that $\iota_i' = \iota_i$, so $\phi = (\iota_2')\inv \phi_Y  \iota_1 \in \Bir_e(\bP^2, \Omega).$
 \end{proof}

\begin{remark}\label{rem:droppingcondition} Note that we can drop the condition that $h^0(Y_{\xi_i}, -K_{Y_{\xi_i}})=1$ at the cost of adding the assumption that $\phi: Y_{\xi_1} \to Y_{\xi_2}$ restricts to a biholomorphism $\phi: U_{\xi_1} \to U_{\xi_2}$ of the interiors. \end{remark}

\begin{remark} The conditions on the surfaces $Y_{\xi_i}$ appearing in the above lemma are satisfied as soon as $m_{\nn} > 2$ for each component $D_{\nn}$ of $D_{\xi_i}$ (since in this case $D_{\xi_i}$ is negative definite); in particular, the set of all $\xi_i \in \CY$ indexing such $Y_{\xi_i}$ is cofinal in $\CY$.
 \end{remark}

We can package the relationship between the biholomorphisms of the surfaces in our system  and elements of $\Bir_{e}(\bP^2, \pm \Omega)$ using more categorical language: 

 \begin{theorem} \label{thm:autisos} Let $\sch$ be the category of schemes over $\C$, and $\prs$ its pro-completion. Given any object $A$ of $\sch$, respectively  $\prs$, let $\Aut_\sch(A)$, respectively $\Aut_{\prs}(A)$, denote its group of automorphisms in the appropriate category. There are group isomorphisms
  \begin{align*} 
\Aut_{\prs}(Y_{\univ}) \cong \Bir_{e}(\bP^2, \pm \Omega) \\ 
 \Aut_{\sch}(U_{\univ}) \cong \Bir_e(\bP^2, \pm \Omega).
\end{align*} 
\end{theorem}

\begin{proof} There are formal equalities (\cite[Prop. VII.2]{may})
 \begin{align*} 
\Hom_{\prs} (Y_{\univ}, Y_{\univ}) = \varprojlim_{\eta \in \CY}\,\, \varinjlim_{\xi \in \CY}\,\, \Hom_{\sch}(Y_{\xi}, Y_{\eta})\\
  \Hom_{\sch} (U_{\univ}, U_{\univ}) = \varprojlim_{\xi \in \CY}\,\, \varinjlim_{\eta \in \CY}\,\, \Hom_{\sch}(U_{\xi}, U_{\eta}). 
 \end{align*} 
Say that $F \in \Aut(Y_{\univ})$ is represented by an element $f_{\xi, \eta} \in \Hom (Y_{\xi}, Y_{\eta})$ if 
\[pr_{\eta}(F) =  [f_{\xi, \eta}] \in    \varinjlim_{\xi' \in \CY}\,\, \Hom(Y_{\xi'}, Y_{\eta}) \] 
where $pr_{\eta}$ is the natural projection map from the inverse limit. 
Any $f_{\xi, \eta}$ representing $F$ is defined by the same birational map, so we obtain a well-defined map
\begin{equation} \label{eq:yautlim} 
\Aut_{\prs}(Y_{\univ}) \to \Bir(\bP^2).
 \end{equation} 
There is an analogous map 
\begin{equation} 
\label{eq:uautlim} \Aut_{\sch}(U_{\univ}) \to \Bir(\bP^2). 
\end{equation}
These maps are group homomorphisms.  

To prove the injectivity of the map from $\Aut_{\prs}(Y_{\univ})$, we use three observations. First, note that the identity $\mathrm{Id} \in \Aut_{\prs}(Y_{\univ})$ is the unique element such that for all $\eta \in \CY$, 
\begin{equation} \label{eqn:dirlim1} pr_{\eta}(\Id)= [\id_{Y_{\eta}}] \in  \varinjlim_{\xi' \in \CY}\,\, \Hom(Y_{\xi'}, Y_{\eta}). \end{equation}
Secondly, note that for all $\eta \in \CY$ and all $\xi \geq \eta$, the projection map $p_{\xi, \eta}: Y_{\xi} \to Y_{\eta}$ is the image of $\id_{Y_{\eta}}$ under the natural map 
\[ \Hom(Y_{\eta}, Y_{\eta}) \to \Hom(Y_{\xi}, Y_{\eta}), \]
so we have 
\begin{equation} \label{eqn:dirlim2} [\id_{Y_{\eta}}] = [p_{\xi, \eta}] \in \varinjlim_{\xi' \in \CY}\,\, \Hom(Y_{\xi'}, Y_{\eta}).\end{equation}
Finally, note that for any pair $\xi, \eta \in \CY$, an element $f \in \Hom(Y_{\xi}, Y_{\eta})$ is defined by $\id \in \Bir(\bP^2)$ if and only if $\xi \geq \eta$ and $f=p_{\xi, \eta}$. 

Now assume that $F \in \Aut_{\prs}(Y_{\univ})$ is in the kernel of the map to $\Bir(\bP^2)$.  Choose any $\eta \in \CY$, and then choose $\xi \in \CY$ such that $pr_{\eta}(F)$ is represented by an element $f_{\xi, \eta} \in \Hom(Y_{\xi}, Y_{\eta})$. Using the third observation above, we deduce that $f_{\xi, \eta}=p_{\xi, \eta}$. But then we know, by the equivalences of Equations \eqref{eqn:dirlim1} and \eqref{eqn:dirlim2}, that 
\[ pr_{\eta}(F) = pr_{\eta}(\mathrm{Id}) \in  \varinjlim_{\xi' \in \CY}\,\, \Hom(Y_{\xi'}, Y_{\eta}). \]
Since $\eta$ was arbitrary, we conclude that $F=\Id$, so the map $\Aut_{\prs}(Y_{\univ}) \cong \Bir_{e}(\bP^2, \pm \Omega)$ is injective as desired. The injectivity of the map $\Aut_{\sch}(U_{\univ}) \to \Bir(\bP^2)$ is proved completely analogously. 

From Proposition \ref{prop:extendphi} we know that the images of the maps in Equations \eqref{eq:yautlim} and \eqref{eq:uautlim} both contain the subgroup $\Bir_{e}(\bP^2, \pm \Omega)$, so to prove the theorem it only remains to show inclusions the other way. We begin with the automorphisms of $Y_{\univ}$. Given any $F \in \Aut(Y_{\univ})$ and $\alpha \in \CY$, define  \[ R(F, \alpha) = \{ \beta \in \Aut(Y_{\univ}) \mid F \text{ is represented by an element } f_{\alpha, \beta} \in \Hom(Y_{\alpha}, Y_{\beta})\}. \\ 
\]
Note that the set 
\[S_{F}= \{ \alpha \in \CY \mid R(F, \alpha) \neq \emptyset \}\]
 is cofinal in $\CY$ and has the property that if $\beta > \alpha \in S_{F}$ and $\alpha \in S_{F}$, then $\beta \in S_{F}$. 
 
 Now fix $F \in \Aut(Y_{\univ})$ and let $G = F\inv$.  Let $\alpha \in S_F$ be any element satisfying the conditions of Lemma \ref{lemma:compatibility}. Since any $\beta \in R(F, \alpha)$ must satisfy $b_2(Y_{\alpha}) \geq b_2(Y_{\beta})$, the set $R(F, \alpha)$ has at least one minimal element; pick one such element $\beta$. We claim that $f_{\alpha, \beta} \in \Hom(Y_{\alpha}, Y_{\beta})$ is a biholomorphism. Let $\gamma \in \CY$ be any element such that $G$ is represented by an element $g_{\gamma, \alpha} \in \Hom(Y_{\gamma}, Y_{\alpha}). $ Suppose for  contradiction that $f_{\alpha, \beta}$ contracts some divisor in $Y_{\alpha}$. 
 Note that 
\[ f_{\alpha, \beta} \circ g_{\gamma, \alpha}=p_{\gamma, \beta} \in \Hom(Y_{\gamma}, Y_{\beta}),  \]
so in other words $f_{\alpha, \beta} \circ g_{\gamma, \alpha}$ factors as a sequence of blow-ups at a set of distinguished points in $Y_{\beta}$. But then any divisor contracted by $f_{\alpha, \beta}$ must go to a distinguished point in $Y_{\beta}$, so $f_{\alpha, \beta}$ must factor through some blow-up
\[ \begin{tikzcd}
                                                                                       & Y_{\beta'} \arrow[d, "{p_{\beta', \beta}}"] \\
Y_{\alpha} \arrow[r, "{f_{\alpha, \beta}}"'] \arrow[ru, "{f_{\alpha, \beta'}}"] & Y_{\beta_0},                                           
\end{tikzcd}\]
contradicting the minimality of $\beta$. Thus  $f_{\alpha, \beta} \in \Hom(Y_{\alpha}, Y_{\beta})$ is a biholomorphism.  Lemma \ref{lemma:compatibility} then implies that the underlying morphism $f \in \Bir(\bP^2)$ is an element of  $\Bir_{e}(\bP^2, \pm \Omega)$ as desired.

Now suppose $F$ is an automorphism of $U_{\univ}$ defined by some $f \in \Bir(\bP^2)$; it remains to show that $f$ is an element of the distinguished subgroup $\Bir_{e}(\bP^2, \pm \Omega)$. We regard $F \in \Aut(U_{\univ})$ as an element of the inverse limit of Equation \eqref{eq:uautlim}, and choose some element $f_{\alpha, \beta} \in \Hom(U_{\alpha}, U_{\beta})$ representing $F$. Let $g := f\inv$ and consider the birational map 
$g: Y_{\beta} \dashrightarrow Y_{\alpha}.$ The indeterminacy locus of this map must be contained within the boundary divisor $D_{\beta}$, since otherwise there exists $y \in D_{\alpha}$ such that $f(y) \in U_{\beta}$, contradicting the invertibility of $F$. 
Resolve the indeterminacies of $g \in \Bir(Y_{\beta}, Y_{\alpha})$ by successively blowing up such points to obtain a new Calabi-Yau surface $Z$ such that the induced map $g: Z \to Y_{\alpha}$ is regular. Let $C \subset Y$ denote the total transform of $D_{\beta}$ under the map $\pi: Z \to Y_{\beta}$; note that $g(C) = D_{\alpha}$. 
The curve $C$ decomposes as the union of a cycle of $\bP^1s$ (which we write as $\Delta$) and several chains $C_1, \ldots, C_k$ of $\bP^1s$, each of which intersect a component of $\Delta$ once. The cycle $\Delta$ consists of the union of the proper transform of the components of $D_{\beta}$ and curves created by blowing up the corners of $D_{\beta}$. But since $g$ maps $C$ to $D_{\alpha}$, the map $g$ must contract each of the chains $C_i$; if we assume that $\pi: Z \to Y_{\beta}$ is the minimal resolution of $g$, then all such chains are contracted, so we conclude that the map $\pi$ factors as a sequence of corner blow-ups. Hence $Z=Y_{\beta'}$ for some $\beta' > \beta \in \CY$, so we obtain an element $g \in \Hom(Y_{\beta'}, Y_{\alpha})$. 

Now pick $\alpha_1 > \alpha$. Then again using Equation \eqref{eq:uautlim}, we can find  $\beta_1 > \beta'$ such that there is a commuting diagram 
\begin{equation} \label{eqn:Udiag}
\begin{tikzcd}
U_{\alpha_1} \arrow[r, "{f_{\alpha_1, \beta_1}}"]            & U_{\beta_1} \\
U_{\alpha} \arrow[r, "{f_{\alpha, \beta'}}"] \arrow[u, hook] & U_{\beta'} \arrow[u, hook].          
\end{tikzcd}
\end{equation}
We can again argue that after a sequence of corner blow-ups on $Y_{\beta_1}$ (yielding a surface $Y_{\beta_1'}$), we can resolve the map $g \in \Bir(Y_{\beta_1}, Y_{\alpha})$ so that there is a commutative diagram 
\[ 
\begin{tikzcd}
Y_{\alpha_1} \arrow[d] & Y_{\beta_1'} \arrow[d]   \arrow[l, "g"]  \\
Y_{\alpha}                              & Y_{\beta'} \arrow[l, "g"].
\end{tikzcd}
\]
Since $\alpha$ and $\alpha_1$ were arbitrary, we can use this procedure to define an element $G \in \Aut(Y_{\univ})$ induced by  $g \in \Bir(\bP^2)$. We have seen that $g$ must lie in $\Bir_{e}(\bP^2, \pm \Omega)$ and we conclude that $f=g\inv$ does as well.   
\end{proof} 
 
 \begin{remark}\label{rem:maps-on-interiors} Note that the argument in the second half of this proof also implies that for any biholomorphism $\phi: U_{\xi_1} \to U_{\xi_2}$, there exists some $\eta_1>\xi_1$ and $\eta_2>\xi_2$ each given by corner blow-ups 
(so $U_{\eta_i} = U_{\xi_i}$ for $i=1,2)$ 
such that $\phi$ extends to a biholomorphism
 \[ \phi: Y_{\eta_1} \to Y_{\eta_2}. \] (A similar argument also appears in \cite{Kollar}.) If we further assume that the toric model $Y_{\xi_1} \to \bar{Y}_{\xi_1}$ contracts divisors to interior points on at least two components of $\bar{D}_{\xi_1}$, then this implies that the underlying birational transformation is an element of $\Bir_{e}(\bP^2, \pm\Omega)$ (see Remark \ref{rem:droppingcondition}). 
 \end{remark}

\subsubsection{Action on $\perf U_{\univ}$}\label{sec:perfuaction}
Theorem \ref{thm:autisos} implies that $\Bir_{e}(\bP^2, \Omega)$ acts on $\Perf U_{\univ}$ by automorphisms. Since fibrant perfect complexes remain fibrant and perfect under any pushforward by a biholomorphism, there is a canonical lift of this to an action on $\perf U_{\univ}$ by dg automorphisms.

 It is also straightforward to define an action on the limit $\varprojlim \coh U_{\xi}$. One way to do this is to assign, to each $\phi \in \Bir_{e}(\bP^2, \Omega)$, a fixed $\xi_{0}$ for which there is a biholomorphism $Y_{\xi_{0}} \to Y_{\phi_*{\xi_0}}$,  and then define the map $\phi_*$ to be the unique map such that the following diagram commutes: \begin{equation} \label{eq:Uaction} \begin{tikzcd}[
  cells={nodes={minimum height=3em, minimum width=5em, text depth=0.25em}},
  column sep=huge
]
  \varprojlim_{\xi \in \CY} \coh U_{\xi} 
    \arrow[r, "\phi_*"] 
    \arrow[d,] 
  & \varprojlim_{\xi \in \CY} \coh U_{\xi} 
    \arrow[d,]  \\
  \varprojlim_{\xi > \xi_0} \coh U_{\xi}          
    \arrow[r, "\prod_{\xi > \xi_0} (\phi_{\xi,\phi_\xi})_*"]                    
  & \varprojlim_{\xi > \phi_*{\xi_0}} \coh U_{\xi}.    
\end{tikzcd}
\end{equation}
The vertical maps here are the natural restriction maps. 
It is readily verified (e.g., by checking on stalks) that the isomorphism in Equation \eqref{eq:perfu}  intertwines these two actions.

\subsection{Mirror symplectomorphisms}

\subsubsection{Constructions of symplectomorphisms} \white{.} \label{sec:mirror-constructions-exact}

\emph{Construction mirror to $E_\nn$.} Fix $\xi \in \CY$, and suppose that $\nn$ is a ray of the fan for $\xi$. Consider the following symplectomorphisms performed in sequence: 
\begin{itemize}
\item First, a sequence of nodal slides, as in \cite[Section 6.1]{Symington}, for the nodes which lie on the line $\bR \cdot \nn$. Consider an isotopy of bases recording the following moves:  if $m_{-\nn} > 0$, slide the node at position $- i \nn$ to position $-(i+1) \nn$; then slide the node at $\nn$ to $-\nn$; and, if $m_\nn > 1$, slide the node at $(i+1) \nn$ to $i \nn$, for all positive $i$. At the end of this isotopy one has a new symplectic manifold $(M_\xi', \omega_\xi')$ with a Lagrangian almost toric fibration over the new base $B_\xi'$. By a Moser argument (which is written out in detail in \cite[Section 8.3]{Evans}) there is a symplectomorphism
\begin{equation}\label{eqn:nodalslide}({E}^\vee_\nn)^{\mathrm{NS}}_\xi: (M_\xi, \omega_\xi) \to (M_\xi', \omega_\xi').\end{equation}

The bases $B_\xi$ and $B_\xi'$ can be identified outside away from a compact set containing the support of the isotopy, and thus so can the manifolds $M_{\xi}$ and $M_\xi'$. We choose the symplectomorphism in Equation \ref{eqn:nodalslide} to respect this identification (it's then uniquely defined up to compactly supported Hamiltonian isotopy).

\item Second, perform a cut transfer \cite[Definition 2.1]{Vianna-CP2}: instead of using the invariant half-line 
$\bR_{\geq -1} \cdot \nn = \bR_{\leq 1} \cdot (- \nn)$
 for the node at $-\nn$, switch to using $\bR_{\geq 1} \cdot (- \nn)$.
This doesn't change the integral affine structure itself (or the almost-toric fibration), but changes the identification of the base with $\bR^2$ by a piecewise linear transformation. For $\nn = (0,1)^T$, we take the cut transfer which fixes the half-plane with non-negative first coordinate, and applies the shear $\begin{pmatrix} 1 & 0 \\  -1 & 1  \end{pmatrix}$ to the other half-plane; and similarly for a general $\nn$, conjugating as before with a suitable element of $SL_2(\bZ)$. The base diagram obtained after the cut transfer on $B_\xi'$ is exactly $B_{{E_\nn}_\ast \xi}$, so we obtain a symplectomorphism
\begin{equation}\label{eqn:cuttransfer}({E}^\vee_\nn)^{\mathrm{CT}}_\xi:(M_\xi, \omega_\xi) \to (M_{{E_\nn}_\ast \xi}, \omega_{{E_\nn}_\ast \xi}).\end{equation}
\end{itemize}

Define $({E}^\vee_\nn)_\xi: M_\xi \to M_{{E_\nn}_\ast \xi}$ as the composition $({E}^\vee_\nn)^{\mathrm{CT}}_\xi \circ ({E}^\vee_\nn)^{\mathrm{NS}}_\xi.$

\begin{proposition}\label{prop:Eprops} 
\begin{enumerate} 

\item \label{item:prop1} For any choice of the defining data, the map $({E}^\vee_\nn)_\xi$ defined above
is an exact symplectomorphism which preserves the Liouville form at the boundary and maps exact Lagrangians with conical ends to other such Lagrangians. Moreover, $({E}^\vee_\nn)_\xi$ is well defined up to compactly supported Hamiltonian isotopy within the space of such symplectomorphisms. 

\item \label{item:prop2} For $\eta \in \CY$ with $\eta \geq \xi$, there exist compatible representatives for $(E^\vee_\nn)_\eta$ and $(E^\vee_\nn)_\xi$,  in the sense that $(E^\vee_\nn)_\eta: M_\eta \to M_{{E_\nn}_\ast \eta}$ restricts to  $(E^\vee_\nn)_\xi$ on $M_\xi \subset M_\eta$.
\end{enumerate}
\end{proposition} 

\begin{proof} 

In a collar neighbourhood of the boundary, where it is just given by the formal cut transfer, $(E^\vee_\nn)_\xi$ is the suspension of a strict contactomorphism of the boundary (with the contact form given by restricting our choice of Liouville form); the first claim follows. The only ambiguity in the definition is in the first symplectomorphism $({E}^\vee_\nn)^{\mathrm{NS}}_\xi$. The fact that two models for the nodal slide are related by compactly supported Hamiltonian isotopy is implicit in the description in Symington's work \cite{Symington} and follows from the fact that one can interpolate between any two choices of data defining the nodal slides and resulting family of symplectic forms used in the Moser argument.

One can check that (\ref{item:prop2}) holds using the description of the embedding $M_\xi \into M_{\eta}$ in terms of almost toric fibrations from Section \ref{sec:almost-toric-fibration-initial}: the nodal slide and cut transfer operations on the base $B_{\xi}$ can be restricted from the operations applied to the base  $B_\eta$.
\end{proof} 
\emph{Constructions mirror to linear maps.} 
Suppose $A \in \GL_2(\bZ) $. Then for any $\xi \in \CY$, there is an obvious symplectomorphism
$$
{A}^\vee_\xi: M_\xi \to M_{A_\ast \xi}
$$
induced by the map on integral affine bases, choosing cut-offs so that $A$ maps $\partial B_\xi$ to $\partial B_{A_\ast \xi}$. 

\begin{proposition} Given any  $A \in \GL_2(\bZ)$, the symplectomorphism ${A}^\vee_\xi$ satisfies Properties (\ref{item:prop1}) and (\ref{item:prop2}) stated in Proposition \ref{prop:Eprops}. \end{proposition} 
\begin{proof}
	 This is clear from the definition. Note that in this case ${A}^\vee_\xi: M_\xi \to M_{A_\ast \xi}$ is defined on the nose, so we do not have the ambiguity of compactly supported Hamiltonian isotopies.
 \end{proof} 

\subsubsection{Group action}

\begin{lemma}\label{lem:relations}
Suppose that we are given a relation  in $\Bir_e (\bP^2, \pm \Omega)$ between elements of $GL_2(\bZ)$ and birational transformations of the form $E_\nn$. 
Then, whenever defined, the corresponding symplectomorphisms of the forms  $(E^\vee_\nn)_\xi$ and $A^\vee_\eta$  satisfy the same relation up to compactly supported Hamiltonian isotopy (here $\xi, \eta \in \CY$ may vary for different $\nn$ or $A$).

This is compatible with the partial order: 
suppose we are given explicit toric models $\xi' \geq \xi$, $\eta' \geq \eta$, with $\xi$ and $\eta$ varying as above, 
and we work with representatives of  the $(E^\vee_\nn)_{\xi'}$ and $A^\vee_{\eta'}$ which are compatible with  the $(E^\vee_\nn)_{\xi}$ and $A^\vee_{\eta}$. 
Then the same compactly supported Hamiltonian isotopy will give the relation between the  $(E^\vee_\nn)_{\xi'}$ and $A^\vee_{\eta'}$.
\end{lemma}

\begin{proof}
Relations for $\Bir_e (\bP^2, \Omega) = \langle \SL_2 (\bZ), E \rangle$ are known from \cite{Blanc}. The only one which is not straightforward to check is the so-called `$A_2$ cluster relation', which is $P^5 = \Id$ in the notation from \cite{Blanc}. This is carefully done in \cite[Section 6.2.2]{HK2}. 

It remains to add in $\begin{pmatrix} -1 & 0 \\ 0 & 1 \end{pmatrix}$. For any $\phi \in  \Bir_e (\bP^2, \Omega)$, there exists a unique $\phi' \in \Bir_e (\bP^2, \Omega)$ such that
$$
\begin{pmatrix} -1 & 0 \\ 0 & 1 \end{pmatrix} \cdot \phi = \phi' \cdot \begin{pmatrix} -1 & 0 \\ 0 & 1 \end{pmatrix}. 
$$
Moreover, from the constructions it is immediate that 
$$
 \begin{pmatrix} -1 & 0 \\ 0 & 1 \end{pmatrix}^\vee_\xi \cdot {\phi}^\vee_\eta = (\phi')^\vee_\xi \cdot {\begin{pmatrix} -1 & 0 \\ 0 & 1 \end{pmatrix}}^\vee_\theta
$$
up to compactly supported Hamiltonian isotopy, for compatible choices of $\xi, \eta$ and $\theta$ in $\CY$. The claim about the compatibility with the partial order on $\CY$ is then also immediate from the definitions.
\end{proof}

Let $\theta$ denote our Liouville one-form on $M_\univ$. Let  $\Symp_e (M_\univ)$ be the group of exact symplectomorphisms of $M_\univ$, i.e.~diffeomorphisms $\varphi: M_\univ \to M_\univ$ such that $\varphi^\ast \theta = \theta + df$ for some compactly supported smooth function  $f$.  Let $\Ham_c (M_\univ)$ denote the normal subgroup of  $\Symp_e (M_\univ)$  consisting of compactly supported Hamiltonian isotopies of $M_\univ$.

\begin{corollary}\label{cor:map-to-symp-Muniv}

There is a well-defined map
$$
\Bir_e (\bP^2, \pm \Omega) \to  \Symp_e M_\univ  / ( \Ham_c M_\univ).
$$

\end{corollary}

\begin{proof}
For a given $\phi \in \Bir_e (\bP^2, \pm \Omega)$, we can pick a factorisation
$$
\phi =   E_{\nn_1} \cdot \ldots \cdot E_{\nn_k} \cdot A
$$
where $A \in \GL_2(\bZ)$. 
Now choose $\xi \in \CY$ such that the intermediate factors
$$\phi_i = E_{\nn_1} \cdot \ldots \cdot E_{\nn_i}   $$
are all biholomorphisms $Y_\xi \to Y_{(\phi_i)_\ast \xi}$ for $i = 1, \ldots, k$ . Now \emph{define} 
$$
\phi_\xi^\vee :=(E_{\nn_1}^\vee)_{\xi_1} \cdot \ldots \cdot (E_{\nn_k}^\vee)_{\xi_k} \cdot  A^{\vee}_{\xi_{k+1}}
$$
where $\xi_i = ( \phi_{i-1})_\ast \xi$ for $i=1, \ldots, k$. 
Suppose now we had picked a different factorisation of $\phi$; a priori $\xi$ must be replaced with some $\eta \in \CY$. For any collection of $A_2$ cluster relations used to relate the two factorisations, there exists some $\theta \in \CY$ dominating both $\xi$ and $\eta$ such that all intermediate factorisations are also realised by biholomorphisms. Now by Lemma \ref{lem:relations}, the two candidates for $\phi^\vee_\theta$ differ by a compactly supported Hamiltonian isotopy of $M_\theta$, say $\rho$.
Moreover, suppose we're given $\theta' \geq \theta$. Then the two factorisations again give different candidates for $\phi^\vee_{\theta'}$. Moreover, for each of the factorisations, we can assume that we use the same representatives for the factors of the form $(E_{\nn_i}^\vee)_{\xi_i}$ or $(E_{\nn_i}^\vee)_{\xi_i'}$ as we did initially. 
Then by Lemma \ref{lem:relations}, the two candidates for $\phi^\vee_{\theta'}$ also differ by $\rho$, now viewed as a compactly supported Hamiltonian isotopy of $M_{\theta'}$. 
Thus we get a symplectomorphism $\phi^\vee$ of $M_\univ$, well-defined up to a compactly supported Hamiltonian isotopy. Exactness of $\phi^\vee$  follows from the exactness of the $\phi^\vee_\xi$. 
\end{proof}

\subsubsection{Mirror symmetry}

\begin{proposition}\label{prop:mirror-symplecto-in-system}
Suppose $\phi \in \Bir_e (\bP^2, \pm \Omega)$ and that $\xi \in \CY$ resolves it. 
Then $\phi^\vee_\xi$ induces well-defined $A_\infty$ functors (up to $A_\infty$ homotopy)
$$\cW(M_\xi, \frak{f}_\xi) \to \cW(M_{\phi_\ast \xi}, \frak{f}_{\phi_\ast \xi}) \qquad \text{and} \qquad \cW(M_\xi ) \to \cW(M_{\phi_\ast \xi}).$$

This is compatible with the partial ordering: for $\eta \geq \xi$, the following diagrams commute up to $A_\infty$-homotopy:
$$
\xymatrix{
\cW(M_\xi, \frak{f}_\xi) \ar[r]^-{\phi^\vee_\xi }  \ar[d]_{} &  \cW(M_{\phi_\ast \xi}, \frak{f}_{\phi_\ast \xi}) \ar[d]_{} \\
\cW(M_\eta, \frak{f}_\eta) \ar[r]^-{\phi^\vee_\eta}  &  \cW(M_{\phi_\ast \eta}, \frak{f}_{\phi_\ast \eta}) \\
}
\qquad
\xymatrix{
\cW(M_\xi ) \ar[r]^-{\phi^\vee_\xi }  &  \cW(M_{\phi_\ast \xi} ) \\
\cW(M_\eta ) \ar[r]^-{\phi^\vee_\eta} \ar[u]_{}   &  \cW(M_{\phi_\ast \eta}) \ar[u]_{}  \\
}
$$
where the vertical maps are the ones constructed in Section \ref{sec:mirror-system-compatibility}. 

Moreover, under the identifications of Theorem \ref{thm:original-hms}, the action on the wrapped category is mirror to pushing forward coherent sheaves, i.e.~the following diagram commutes up to $A_\infty$ homotopy:

\begin{equation} \label{eqn:HMS-diagram}
\xymatrix{
\cW(M_\xi ) \ar[r]^-{\phi^\vee_\xi }  \ar[d]_-\simeq &  \cW(M_{\phi_\ast \xi} ) \ar[d]_-\simeq \\
\coh  U_\xi   \ar[r]_-{\phi_\ast}  & \coh U_{\phi_\ast \xi}.
}
\end{equation}
\end{proposition}

\begin{proof}
As with the proof of Proposition \ref{prop:hms-inclusion-compatibility}, by \cite[Theorems 1.2 and 1.5]{Genovese}, to show that the diagrams of $A_\infty$ categories commute up to $A_\infty$-homotopy, it is enough to show that all the the corresponding diagrams of homotopy categories commute. (Compared with before, this just uses additionally that the induced functors $H^0(\phi^\vee_\xi)$ (between either a pair of stopped categories or a pair of unstopped ones) and $H^0(\phi_\ast): D(U_\xi) \to D(U_{\phi_\ast \xi})$ are fully faithful.) Note that $H^0(\phi_\ast)$ is often also just denoted $\phi_\ast$; for clarity in this proof we keep track of the difference.

The first part of the proposition largely follows from the constructions in Section \ref{sec:mirror-constructions-exact}.
 In particular, we have that $\phi^\vee_\xi$ maps $\frak{f}_\xi$ to $\frak{f}_{\phi_\ast \xi}$.
To define maps of Fukaya categories, recall that we're working with the $\bZ$ graded versions of them. Any $\xi \in \CY$ is dominated by $\eta \in \CY$ such that $H^1(M_\eta; \bZ) = 0$. In that case, there is a unique choice of trivialisation for 
$(\Lambda^2 T^\ast M_\eta)^{\otimes 2}$ 
(or $(\Lambda^2 T^\ast M_{\phi_\ast \eta})^{\otimes 2}$), which 
$\phi^\vee_\eta$ preserves.
 Thus $\phi^\vee_\xi$ preserves our choice of trivialisation in general, by restriction. There is a preferred choice of graded lift for $\phi^\vee_\xi$: the one which preserves the gradings on the `reference' Lagrangian zero sections, which are mapped to each other.
 The $A_\infty$ functor definitions are now standard, and compatibility with the partial ordering is automatic.

We now need to establish the mirror claim. First, assume that $\phi^\vee_\xi$ is a composition of maps of the form $(E_{\nn_i}^\vee)_{\xi_i}$ (in the notation of the proof of Corollary \ref{cor:map-to-symp-Muniv}, $A=\Id$). Then the additional assmptions on $\xi$ are not required, and the claim follows by inspecting the proof of \cite[Theorem 6.1]{HK2}: for each $E_{\nn_i}$, the maps
$$
H^0( (E_{\nn_i}^\vee)_{\xi_i} ):  \,
H^0( \cW(M_{\xi_i}, \frak{f}_{\xi_i}) )
\to  H^0( \cW(M_{ E_{\nn_i \ast} {\xi_i} }, \frak{f}_{E_{\nn_i \ast} {\xi_i} } ) )
\, \, 
\text{and} \, \,
H^0( \cW(M_{\xi_i}) ) \to  H^0( \cW(M_{ E_{\nn_i \ast} \xi_i} ))
$$
are known to be mirror to the pushforward maps $H^0(E_{\nn_i})_\ast$ on $D (Y_{\xi_i})$ and $D (U_{\xi_i})$, by carefully translating the relevant moves between the Lefschetz and almost-toric models for $M_{\xi_i}$. 

It remains to analyse the action of $A^\vee_\xi$ for $A \in \GL_2(\bZ)$.  First, suppose that we are given a $(-2)$ curve $C \subset U_\xi$. By  \cite[Propositions 5.1 and 5.2]{HK2},  up to nodal slides and cut transfers, the mirror to $i_\ast \cO_C (-1)$ is a visible Lagrangian sphere as in  Lemma \ref{lem:mirrors-of-(-2)-curves-blowup}, and the mirrors to $i_\ast \cO_C (k)$ for other $k \in \bZ$ are its Lagrangian translates. These objects immediately have the correct images under the action of $A^\vee_\xi$.  Second, by \cite[Lemma 4.20]{HK2}, we have an explicit description of how the HMS isomorphism descends to an isomorphism of $K$-theories:
$$
K(U_\xi) \simeq K_0 (H^0 \cW (M_\xi)) \simeq H_2(M_\xi, \partial M_\xi).
$$ 
From this, one can readily check that our mirror claim holds at the level of the $K$-theory of $\cW(M_\xi)$. 

To conclude, we use the fact that $A^\vee_\xi$ induces a equivalence of $A_\infty$ categories: $(A^\vee_\xi)_\ast: \cW(M_\xi, \mathfrak{f}_\xi) \to \cW(M_\eta, \frak{f}_{A_\ast \xi})$. 
Consider the mirror (homotopy-category level) equivalence $D(Y_\xi) \to D(Y_{A_\ast \xi})$. We want to compare this to the pushforward autoequivalence $A_\ast$. 
To do this, we use Uehara's work on autoequivalences of smooth projective surfaces \cite[Theorems 6.6 and 6.8]{Uehara-trichotomy}, the conclusion of which in the case of a log Calabi-Yau surface is summarised in \cite[Theorem 2.10]{HK2}.

To do this, we use a variant of \cite[Proposition 2.14]{HK2}. Given an (abstract) log CY2 surface $(Y,D)$, this proposition characterises autoequivalences $\psi   \in \Auteq D(Y)$ which (i) induce the identity on $K(Y)$, and (ii) such that for each $(-2)$ curve $C \subset Y$, $\psi$ fixes $i_\ast \cO_C$ and $i_\ast \cO_C(-1)$ as objects of $D(Y)$: 
any such $\psi$ must be the identity up to a finite-group ambiguity, which vanishes as soon as $\pi_1(Y \backslash D) = 0$. 
In our case, by passing to $\xi' > \xi$, we can assume without loss of generality that $\pi_1(Y_\eta \backslash D_\eta) = 0$. 
The anticanonical divisor $D_\xi$ may contain some $(-2)$ components, which necessarily form a cycle or union of chains, disjoint from all the other $(-2)$ curves in $Y_\xi$. 
We don't know the action of $A^\vee_\xi$ on the mirror to the subcategory of coherent sheaves on $Y_\xi$ with support of the union of these curves. 
However, we can constrain this action abstractly by using Ishii-Uehara and Uehara's work \cite{Ishii-Uehara, Uehara-elliptic, Uehara-trichotomy} (summarised in \cite[Lemma 2.1 and Theorem 2.10]{HK2}, see also the discussion in Section 2.4 \textit{ibid}). 
Using this, the proof of \cite[Proposition 2.14]{HK2}, adapted at the end, shows that that the action of $H^0 ( A^\vee_\xi)$ on $H^0 (\cW(M_\xi, \frak{f}_{\xi}))$ agrees with the pushforward $H^0(A_\ast)$ on $D (Y_\xi)$ up to a composition of spherical twists in $(-2)$ components of $D_\xi$ (should there be any) and a tensor with a line bundle which vanishes when pulled back to $U_\xi$. Finally, forgetting the stop, we get that the action of $H^0 (A^\vee_\xi )$ on $H^0 (\cW(M_\xi))$ agrees with the pushforward $H^0(A_\ast)$ on $D (U_\xi)$.
\end{proof}

\begin{remark}
We also expect the action $\phi^\vee_\xi: \cW(M_\xi, \mathfrak{f}_\xi) \to \cW(M_\eta, \frak{f}_{\phi_\ast \xi})$ to be mirror to $\phi_\ast : \coh (Y_\xi) \to \coh (Y_{\phi_\ast \xi})$. In order to check this, one would need:
\begin{itemize}
\item[(a)] To rule out spherical twists in components of $D_\xi$; this can be achieved by passing to $\eta \geq \xi$ such that $D_\eta$ contains no $(-2)$ curves, and considering the commutative diagram
$$
\xymatrix{
D (Y_\eta) \ar[r]^-{\phi_\ast }  \ar[d]_-\simeq &  D (Y_{\phi_\ast \eta}) \ar[d]_-\simeq \\
D (Y_\xi)   \ar[r]_-{\phi_\ast}  & D (U_{\phi_\ast \xi})
}
$$
and its mirror.

\item[(b)] To carefully identify the mirror to the action of the whole Picard group $\Pic (Y_\xi)$ by tensor on $D (Y_\xi)$. Following \cite{Abouzaid-toric, Hanlon-Hicks, HK2}, there is a clear expectation: these should correspond to Lagrangian translations in sections of $\pi_\xi$ with prescribed boundary behaviour. As the action on the wrapped category $\cW(M_\univ)$, below, is our principal objective, we do not pursue this here. 
\end{itemize}
\end{remark}

\begin{theorem}\label{thm:mirror-map}
For any $\phi \in \Bir_e(\bP^2, \pm \Omega)$, its mirror $\phi^\vee \in  \Symp_e M_\univ  / ( \Ham_c M_\univ)$ gives a well-defined $A_\infty$ automorphism of $\cW(M_\univ)$, which we denote by $\phi^\vee$. 
Under the quasi-isomorphisms  of Corollary \ref{cor:universal-exact-hms}, this is mirror to pushing forward coherent sheaves, i.e.~the following diagram commutes up to $A_\infty$ homotopy:
$$
\xymatrix{
\perf U_{\univ}   \ar[d]^-\simeq \ar[r]_-{\phi_\ast}  &  \perf U_{\univ}  \ar[d]^-\simeq\\
\cW(M_{\univ})   \ar[r]^-{\phi^\vee}  & \cW(M_{\univ})    \\
}
$$
In particular, our map $$\Bir_e (\bP^2, \pm \Omega) \to \Symp_e M_\univ / (\Ham_c M_\univ)$$ is injective. 
\end{theorem}

\begin{proof} We can define a strict action of $\Bir_{e} (\bP^2, \Omega)$ on $\cW(M_{\univ})$ using the identification of $\cW(M_{\univ})$ with $\varprojlim \coh U_{\xi}$ that we used in Definition \ref{def:asidecats} together with the action of $\Bir_{e}(\bP^2, \Omega)$ on $\varprojlim \coh U_{\xi}$ defined in Section \ref{sec:perfuaction}. The commutativity of the diagram then follows from Proposition \ref{prop:mirror-symplecto-in-system} (with our chosen model we can assume that the diagram in Equation \eqref{eqn:HMS-diagram} is always strictly commutative) along with the fact that the quasi-isomorphism of Equation \eqref{eq:perfu} is $\Bir_{e}(\bP^2, \Omega)$ equivariant. The injectivity claim follows from the injectivity of the $\Bir_e(\bP^2, \pm \Omega)$ action.
\end{proof}

\begin{remark}
We also get an $A_\infty$ autoequivalence of $\cF^\to (w_\un)$ (however, as noted earlier, this is more contrived), which we expect to be mirror to the action of $\phi$ on $\coh Y_\univ$ by pushforward. 
Alternatively, instead of using the almost-toric fibration, we could construct a mirror  $\tilde{A}_\xi $ to $A\in GL_2(\bZ)$  in terms of the Lefschetz fibration $w_\xi: M_\xi \to \bC$. This would be a fibred symplectomorphism, given by combining an action of a dihedral group element on the fibre near infinity(rotation and possibly spinning in a symmetry axis) together with an automorphism of the base of $w_\xi$. This would allow ready control over the action on $\cW(M_\xi, \frak{f}_{\xi})$, albeit at the cost of being much more geometrically cumbersome.
\end{remark}

We also explicitly record what we get at the level of a fixed log Calabi-Yau surface. Suppose $U$ is an open log Calabi-Yau surface with split mixed Hodge structure; there exists a log Calabi-Yau pair $(Y_\xi,D_\xi )$ such that $U\simeq U_\xi$. The cleanest statement holds when we have rigidity of the automorphisms of $U$: assume that the toric model $p_{\xi}: Y_{\xi} \to \bar{Y}_{\xi}$  contracts divisors living over at least two interior points of $\bar{D}_{\xi}$ on components indexed by linearly independent elements of $(\Z^2)_{\prim}$.  Note that the whole graded symplectic mapping class group $\pi_0 \Symp^{\gr} (M_\xi)$ acts on $\cW(M_\xi)$, by the same arguments as \cite[Corollary 2.4]{Keating-Smith}. 

\begin{corollary}\label{cor:mirror-symplecto-fixed}
Suppose $U \simeq U_{\xi} $ is an open log Calabi-Yau surface satisfying the condition above, with mirror $M  = M_\xi$.
 Fix the HMS isomorphism $\cW(M) \simeq \coh U$. Then there exists an injective map
\begin{eqnarray}
\Aut ( U )& \to & \pi_0 \Symp^{\gr} (M) \\  
\phi & \mapsto & \phi^\vee 
\end{eqnarray}
such that for each $\phi \in \Aut (U)$,  we have a commutative diagram up to $A_\infty$ homotopy:
$$
\xymatrix{
\cW(M ) \ar[r]^-{\phi^\vee }  \ar[d]_-\simeq &  \cW(M  ) \ar[d]_-\simeq \\
\coh  U  \ar[r]_-{\phi_\ast}  & \coh U 
}
$$
\end{corollary}

\begin{proof}
As noted in Remark \ref{rem:maps-on-interiors}, any $\phi \in \Aut(U)$ is defined by an element of $\Bir_{e}(\bP^2, \pm \Omega)$, and extends to a biholomorphism
\[ \phi: Y_{\eta_1} \to Y_{\eta_2} \]
for some $\eta_1$ and $\eta_2 = \phi_* \eta_1$, where $Y_{\eta_1}$ and $Y_{\eta_2}$ are each constructed by a sequence of corner blow-ups on $D_{\xi} \subset Y_{\xi}$. We can choose a factorization  $\phi= \phi_k \circ \cdots  \circ \phi_1$ 
by elements in our preferred generating set so that the intermediate maps  $\phi_\ell \circ \cdots \circ \phi_1$ are regular on $Y_{\eta_1}$ for all $1 \leq \ell \leq k$, and define a symplectomorphism $\phi^\vee: M_{\eta_1} \to M_{\phi_*\eta_1}$ using this factorization. We can then appeal to Proposition \ref{prop:mirror-symplecto-in-system}: since corner blow-ups don't change the interior, we have $U_{\eta_1} =U_{\phi_* \eta_1} = U$ and $M_{\eta_1} = M_{\phi_* \eta_1} =M$, so the diagram in Equation \eqref{eqn:HMS-diagram} agrees with the diagram in the statement of the theorem. 

It remains to check only that the resulting element of the symplectic mapping class group is independent of the choice of $\eta_1$ and the factorization of $\phi$ but this follows from the same proof that appears in Corollary \ref{cor:map-to-symp-Muniv} -- one notes that given two choices $\eta_1, \eta_1'$, we can choose the $\theta$ dominating them both such that $U_{\theta} = U_{\xi}$. 
\end{proof}

\begin{remark} \label{rem:non-symplectos}

We can't expect automorphisms of an open Calabi--Yau surface to be mirror to symplectomorphisms in complete generality: for instance, in the case of $U = (\bC^\ast)^2$, the compact Fukaya category of the mirror $M = T^\ast T^2$ is generated by the zero-section $T^2$, which as a Lagrangian brane can be equipped with $(\bC^\ast)^2$'s worth of local systems. Then multiplication by an element of $ (\bC^\ast)^2$ on $U$ is mirror to an autoequivalence of the Fukaya category which transforms the local systems on this $T^2$.  
\end{remark}

\subsection{Application: automorphisms of the open cubic surface } 
If $(Y_{\xi}, D_{\xi})$ are such that $D_\xi$ contains no $(-1)$ curves, and that $D_\xi^2 \leq 2$, then any automorphism $\phi $ of $U_\xi$ extends to an automorphism of $Y_\xi$. In this case, as there is a Torelli theorem for automorphisms of $Y_\xi$ \cite{GHK2}, the mirror subgroup of $\pi_0 \Symp^\gr M_\xi$ which we construct will embed in an integral linear group. On the other hand, when these additional hypotheses are not satisfied, one can get some interesting (and non-linear) dynamical behaviours. We illustrate this with the worked example of an open cubic surface.

 Consider the log Calabi-Yau pair $(Y_C,D_C)$ obtained by blowing up twice at the distinguished point on each of the three toric boundary divisors of $\bP^2$, and let $U_C$ denote the complement $Y_C\setminus D_C$. This is a member of the family of Markov cubics studied in \cite{EH} and \cite{CL}; in the former, it is shown that $\Aut(U_C)$ is generated by two subgroups: a finite index subgroup (of index at most 24)
\[ G \simeq \Z/2\Z \star \Z/2 \Z \star \Z/2 \Z.\]
and a subgroup $H \simeq \Aut(Y_C, D_C)$
(where the right-hand side consists of the set of automorphisms of $Y_C$ fixing $D_C$ as a set). By Corollary \ref{cor:mirror-symplecto-fixed}, there is an inclusion
\[ \Aut(U_C) \into  \pi_0 \Symp^{\gr}(M_C)\]
where $M_C$ is the mirror to $Y_C$; in particular, we can conclude:
\begin{proposition}\label{prop:cubic-incl} There is an inclusion
\[ \Z/2\Z \star \Z/2 \Z \star \Z/2 \Z \into \pi_0 \Symp^{\gr}(M_C). \]
\end{proposition}

We can understand the generators of the free product geometrically on both sides of the mirror correspondence. The description on the algebraic side is found in \cite{EH}: after  fixing any map from $Y_C$ to a degree three hypersurface in $\bP^3$, the three order two generators are given by the reflections $r_{p_1}$, $r_{p_2}$, and $r_{p_3}$ about the corner points $p_1, p_2$,  and $p_3$ (defined, for each $i$, so that the set of points $x, r_{p_i}(x)$, and $p_i$ are co-linear in $\bP^3$).

\begin{figure} \label{fig:cubicaut}

\begin{tikzpicture}
\draw[thick, ->] (0,0) -- (1.5,0) node[right] {};
\draw[thick, ->] (0,0) -- (0,1.5) node[above] {};
\draw[thick, ->] (0,0) -- (-1.5,-1.5) node[below] {};
\draw[thick, lightgray, ->] (0,0) -- (0,-1.5) node[right] {};
\node at (0,0) {$\cdot$};
\node at (.5,0) {$\times$};
\node at (1,0) {$\times$};
\node at (0,.5) {$\times$};
\node at (0,1) {$\times$};
\node at (-.5,-.5) {$+$};
\node at (-1,-1) {$+$};
\draw[-> ,thick] (2.2, 0) arc (260 : 280 : 4);
\node at (3,.75) {$E^2$};

\begin{scope}[xshift=6cm, yshift=0];
\draw[thick, ->] (0,0) -- (1.5,0) node[right] {};
\draw[thick, ->] (0,0) -- (0,1.5) node[above] {};
\draw[thick, ->] (0,0) -- (-1.5,1.5) node[below] {};
\draw[thick, lightgray, ->] (0,0) -- (0,-1.5) node[right] {};
\node at (0,0) {$\cdot$};
\node at (.5,0) {$\times$};
\node at (1,0) {$\times$};
\node at (0,-.5) {$\times$};
\node at (0,-1) {$\times$};
\node at (-.5,.5) {$+$};
\node at (-1,1) {$+$};
\draw[-> ,thick] (2.2, 0) arc (260 : 280 : 4);
\node at (2.9,1) {$\begin{pmatrix}1 & 0 \\ 0 &-1 \end{pmatrix}$} ;
\end{scope}

\begin{scope}[xshift=12cm, yshift=0cm]
\draw[thick, ->] (0,0) -- (1.5,0) node[right] {};
\draw[thick, lightgray, ->] (0,0) -- (0,1.5) node[above] {};
\draw[thick, ->] (0,0) -- (-1.5,-1.5) node[below] {};
\draw[thick, ->] (0,0) -- (0,-1.5) node[right] {};
\node at (0,0) {$\cdot$};
\node at (.5,0) {$\times$};
\node at (1,0) {$\times$};
\node at (0,.5) {$\times$};
\node at (0,1) {$\times$};
\node at (-.5,-.5) {$+$};
\node at (-1,-1) {$+$};
\end{scope}

\end{tikzpicture}

\caption{The sequence of moves recording the factorization of the reflection $r_{p_2}$; the rays illustrate the divisors appearing in the relevant compactifications of $U_{C}$. 
On a Zariski open set, this automorphism takes a point in the cubic to the other point colinear with the intersection point $y_{p_2}$ of the divisors corresponding to the rays $(1,0)$ and $(-1,-1)$ in the above diagram. In order to resolve $r_{p_2}: Y_C \dashrightarrow Y_C$, we have to blow up $p_2$, thus creating a new divisor corresponding to the grey ray in the first diagram; note that this does not change the interior $U_C$ nor the the mirror $M_C$.}
\end{figure}

To find the mirror symplectomorphisms we factorize each reflection.  Let $\iota: \cst \into \bP^2$ be the standard embedding $(x,y) \mapsto [x: y: 1]$ and let $p_2$ be the pre-image of the point $[0:1:0] \in \bP^2$ under the blow-up map $Y_{C} \to \bP^2$. One can calculate explicitly (using e.g. the map $Y_{C} \to \bP^3$ appearing in the proof of  Lemma 3.10 in \cite{CL}) that the birational map $\iota\inv \circ r_{p_2} \circ \iota \in \Bir(\cst)$ is given in coordinates 
\[ (x,y) \mapsto (x, (1+x)^2 y \inv)\]
 which factorizes as
\[ \iota\inv \circ r_{p_2} \circ \iota = \begin{pmatrix} -1 & 0 \\ 0 & 1 \end{pmatrix} \circ E^2.\] 
The other two generators can be described analogously. See Figure \ref{fig:cubicaut} for an illustration of the symplectomorphism mirror to $r_{p_2}$ (more precisely, we exhibit the relevant sequence of moves of the base diagrams of the almost toric fibrations).

\bibliography{bib}{}
\bibliographystyle{alpha}

\end{document}